\documentclass[runningheads]{llncs}
\usepackage{amsmath,amssymb}
\usepackage{mathtools}
\usepackage{hyperref}
\usepackage{xcolor}
\usepackage[capitalize,nameinlink,sort]{cleveref}
\usepackage{graphicx}
\DeclareMathOperator{\rep}{rep}

\newcommand{\seq}[3]{({#1}_{#2})_{#2 \ge #3}}

\usepackage{float}

\begin{document}
\title{Combinatorics on words and generating Dirichlet series of automatic sequences
%\thanksw{}
}
\titlerunning{Combinatorics on words and automatic Dirichlet series}
\author{Jean-Paul Allouche\inst{1}\orcidID{0000-0002-9060-0784} 
\and
Jeffrey Shallit\inst{2}\orcidID{0000-0003-1197-3820} 
\and
Manon Stipulanti\inst{3}\orcidID{0000-0002-2805-2465}}
\authorrunning{J.-P. Allouche, J. Shallit, and M. Stipulanti}
\institute{CNRS, IMJ-PRG, Sorbonne, Paris, France \\
\email{jean-paul.allouche@imj-prg.fr}\\ 
\and
School of Computer Science, University of Waterloo, Waterloo, Canada \\
\email{shallit@uwaterloo.ca} \\
\and
Department of Mathematics, ULiège,
Liège, Belgium\\
\email{m.stipulanti@uliege.be}}
\maketitle              % typeset the header of the contribution
\begin{abstract}
%The abstract should briefly summarize the contents of the paper in 15--250 words.
Generating series are crucial in enumerative combinatorics,
analytic combinatorics, and combinatorics on words. Though it 
might seem at first view that generating Dirichlet series are less 
used in these fields than ordinary and exponential generating series, 
there are many notable papers where they play a fundamental role, as 
can be seen in particular in the work of Flajolet and several of his 
co-authors. In this paper, we study Dirichlet series of integers with 
missing digits or blocks of digits in some integer base $b$; i.e., 
where the summation ranges over the integers whose 
expansions form some language strictly included in the set of all 
words over the alphabet $\{0, 1, \dots, b-1\}$ that do not begin with 
a $0$. We show how to unify and extend results proved by Nathanson 
in 2021 and by K\"ohler and Spilker in 2009.
En route, we encounter several sequences from Sloane's On-Line Encyclopedia of Integer Sequences, as well as some famous $b$-automatic sequences or $b$-regular sequences.
We also consider a specific sequence that is not $b$-regular.

\bigskip

\end{abstract}
\section{Introduction}

Combinatorics (and in particular combinatorics on words) frequently 
uses generating functions. The point of generating functions is to
``synthesize'' or ``concentrate'' the properties of 
the sequence of its coefficients in a single function. There are three main types of 
generating functions associated with a sequence 
$(s_n)_{n \ge 0}$:
\begin{itemize}
\item[ ]{$\star$} the ordinary generating function: 
\ \ \ \ \ \ $\displaystyle\sum_{n \geq 0} s_n x^n$;

\item[ ]{$\star$} the exponential generating function: 
\ \ $\displaystyle\sum_{n \geq 0} \frac{s_n}{n!} x^n$;

\item[ ]{$\star$} the Dirichlet generating  function: 
\ \ \ \ \ \ $\displaystyle\sum_{n \geq 1} \frac{s_n}{n^z}$, where $z \in {\mathbb C}$.
\end{itemize}

Note that these series can be considered either as ``formal'' series
or as complex (or real) functions only defined in some {\it ad hoc} 
domain. An excellent source for generating functions is Wilf's book
\cite{Wilf-2006} (also available online).
At first glance, it might seem that ordinary and exponential generating 
series are more commonly used in (enumerative) combinatorics, while 
Dirichlet series are more used in number theory. However, this impression is certainly biased due, in particular, to the most famous such series,
namely the Riemann zeta function $\zeta: z\mapsto \sum_{n\ge 1} 1/n^z$. To get easily 
convinced that generating Dirichlet series are \emph{also} used in combinatorics, 
it suffices to mention the strong link between ordinary 
generating functions and generating Dirichlet functions through the
Mellin transform (see, for example, the paper of Wintner~\cite{Wintner-1947}, 
who cites the 1859 pioneering paper of
Riemann\footnote{See a transcription at 
\url{https://www.claymath.org/sites/default/files/zeta.pdf}.}), 
% référence précise Monatsberichte der Berliner Akademie, November 1859, und Gesammelte Mathematische Werke, 1. Auflage, 1876, Seite 136—144; 2. Auflage, 1892, Seite 145—153 : citée dans 
%Zu Riemanm Abhandlung „lieber die Anzahl der Primzahlen unter einer gegebenen Grosse"*).(Von Herrn H. von Mangoldt in Aachen.)
% Voir aussi MR1084595.
or the 
many papers of the \emph{virtuosi} Flajolet and co-authors, who make use of 
the Mellin transform. In particular one can consult the ``bible'' 
by Flajolet and Sedgewick \cite{Flajolet-Sedgewick-2009}, from which 
we do not resist quoting the following lines:
\begin{quote}
{\it [...] It is possible to go much further as first shown 
by De Bruijn, Knuth, and Rice in a landmark paper 
\cite{de-Bruijn-Knuth-Rice-1972}, which also constitutes a historic 
application of Mellin transforms in analytic combinatorics. (We refer 
to this paper for historical context and references.)}
\end{quote}

Playing around with harmonic series and Dirichlet series with ``missing 
terms'' in the summation range, we revisited the 1914 paper of Kempner
\cite{Kempner1914}, where he studied the series $\sum' 1/n$; see~\cref{th:Kempner} presented below. Here 
$\sum'$ means that the summation ranges over integers with no occurrence of the
digit $9$ in their base-$10$-expansion. 
The \emph{abscissa of convergence} (whose definition is recalled in~\cref{sec:abscissa}) of the Dirichlet series $\sum' 1/n^z$ is less than or equal to $1$. This abscissa of convergence is worth determining 
     precisely, not only for its own sake, but also because it is related to the singularities of the series.
%In particular, his result implies that the series $\sum' 1/n^z$ converges and moreover that its so-called the abscissa of convergence is less than or equal to $1$.
%Thus the abscissa of convergence of this Dirichlet series is worth determining precisely, not only for its own sake but also because it relates to the singularities of the series.
Actually, it is natural 
to look at series with a more general flavor.  Following this, we 
came across two papers: namely, a 2009 paper by K\"ohler and 
Spilker \cite{Kohler-Spilker-2009}, and a 2021 paper by Nathanson 
\cite{Nathanson2021}, where the abscissa of convergence 
of ``restricted'' Dirichlet series was determined.
Here, ``restricted'' means that 
integers having some digits or some blocks of digits in their 
$b$-ary expansion are excluded from the summation range. 

Within the topic of restricted Dirichlet series, our contribution is threefold.
First, it appears that a classical result (\cref{lem:Titchmarsh}) about the abscissa of convergence of Dirichlet series used in \cite{Kohler-Spilker-2009} provides a (shorter?) proof of results in \cite{Nathanson2021}.
Our first contribution is thus to uncover and give credit to~\cref{lem:Titchmarsh}, and provide the alternative proof of the results in \cite{Nathanson2021} in~\cref{sec:short-proof}.
Second, it turns out that results of K\"ohler and 
Spilker \cite{Kohler-Spilker-2009} can be extended to generalize those of Nathanson \cite{Nathanson2021}.
So, in~\cref{sec:myriad}, we prove results giving the abscissa of convergence of some specific restricted Dirichlet series (see~\cref{thm:abscissa-convergence-blocks-12-89,thm:abscissa-convergence-blocks-12-21,thm: avoiding block of a} below).
In these cases, the abscissa of convergence is equal to $\log_b \lambda$, where $\lambda$ is the dominant root of a certain polynomial with integer coefficients, and $b$ is the base of the considered numeration system.
The examples taken in~\cref{sec:myriad} illustrate how \cref{lem:Titchmarsh} can be used beyond Nathanson's considerations on the one hand, and K\"ohler and Spilker's on the other: in all cases, the key is to \emph{count} words of a prescribed form.
In addition, our examples appear in Sloane's On-Line Encyclopedia of Integer Sequences~\cite{Sloane}.
In~\cref{sec:alternative approach}, we provide yet a different method and show how to recover our results from the general study of the asymptotics of summatory functions of so-called $b$-regular sequences (see~\cref{sec:notation and def} for definitions).
Moreover, with~\cref{thm: general abscissa of convergence automatic sec}, we provide a general statement on the abscissa of convergence of Dirichlet series associated with languages having automatic characteristic sequences (see again~\cref{sec:notation and def} for definitions).
In~\cref{sec:unusual}, we go beyond this setting and, as a third contribution, we show how our initial method can be applied to some non-regular Dirichlet series.
We finish the paper with~\cref{sec:conclusion} where we conclude and expose some related open problems.

As a final comment to this introduction, we would also like to add that related work
is contained in the papers 
\cite{Allouche-Shallit-Skordev-2005,Allouche-Yu-Morin-2024,Birmajer-Gil-Weiner-2016,Burnol-3,Burnol-1,Burnol-2,Cumberbatch,Zeilberger2023,Karki-Lacroix-Rigo-2009,Karki-Lacroix-Rigo-2010,MukherjeeSarkar},
in the papers of Janji\'c cited at the beginning of~\cref{sec:myriad},
and the paper of Noonan and Zeilberger~\cite{Noonan-Zeilberger-1999}.

\subsection{Notation and definitions}
\label{sec:notation and def}

We let ${\mathbb N}$ denote the set of non-negative integers
$\{0, 1, 2, \dots\}$. A subset $Y$ of a set $X$ is called
\emph{proper} if $Y \neq X$.
We let $|X|$ denote the \emph{cardinality} of the finite set $X$.
(Also note that we will use the notation 
$|x|$ to denote the absolute value of the real number $x$. 
There should be no confusion between the two uses of $| \ | $.)

\medskip

Let $\seq{s}{n}{1}$ be a sequence of real numbers. Its
\emph{summatory function} is the sequence $(A(n))_{n \geq 1}$ of its \emph{partial sums} defined by $A(n) = \sum_{i=1}^n s_i$ for each 
$n\ge 1$. (Following the usual convention for a summation over an 
empty set of indices, we set $A(0) = 0$.)
We also warn the reader that most sequences in this text will be indexed starting at $1$ (except $b$-ary expansions as mentioned below).

\medskip

An \emph{alphabet} is a finite set. A (finite) \emph{word} over an alphabet $\Sigma$
is a finite sequence of letters from $\Sigma$.
The \emph{length} of a word is the number of letters it is made of.
The \emph{empty word} is the only word of length $0$.
For all $n\ge 0$, we let $\Sigma^n$ denote the 
set of all length-$n$ words over $\Sigma$.
We let $\Sigma^*$ denote the set of words over $\Sigma$ (including the empty word). It is equipped with the concatenation of words,
which makes it a monoid (called the \emph{free monoid} on $\Sigma$).

\medskip

Let $b\ge 2$ be an integer. If the expansion of an integer $n$ in the 
base-$b$ numeration system is $n = \sum_{0 \leq j \leq k} w_j b^j$,
with $w_j \in [0,b-1]=\{0, 1, \dots, b-1\}$ and $w_k \neq 0$,
then the \emph{base-$b$} or \emph{$b$-ary representation} of $n$, written 
$\rep_b(n)$, is the word $\rep_b(n) = w_k w_{k-1} \cdots w_0$ over the alphabet $[0,b-1]$ 
(we warn the reader to pay attention to the order of indices and to the fact that they end at $0$).
The base-$b$ representation of $0$ is the empty word.
A sequence $\seq{s}{n}{0}$ is \emph{$b$-automatic} if its \emph{$b$-kernel} $\{(s(b^e n+r))_n \ | \ e \geq 0, r \in [0, b^k-1]\}$ is finite.
For more on these sequences, see, for example, the books \cite{Allouche-Shallit-2003,Fogg-2002,Haeseler-2003}.
As automatic sequences take only finitely many values, Allouche and Shallit introduced generalizations to infinite alphabets, called $b$-regular sequences~\cite{Allouche-Shallit-1992}: a sequence $\seq{s}{n}{0}$ is \emph{$b$-regular} if there exist vectors $V,W$ and matrices $M_0,\ldots,M_{b-1}$ such that, if $\rep_b(n) = w_k w_{k-1} \dots w_0$, then $s(n)= V M_{w_k} \cdots M_{w_0} W$ for all $n\ge 0$.
The set of vectors $V,W$ and matrices $M_0,\ldots,M_{b-1}$ form a \emph{linear representation} of 
$\seq{s}{n}{0}$.
\medskip

For two sequences $\seq{u}{n}{0}, \seq{v}{n}{0}$, we recall the notation $u_n \sim v_n$, which means that there exists an integer $n_0$ such that  $v_n \neq 0$ for $n \geq n_0$ and $\lim_{n \to +\infty} \frac{u_n}{v_n} = 1$.
We also write $u_n = \Theta(v_n)$ if there exist positive constants $c_1,c_2$ such that $c_1 v_n \le u_n \le c_2 v_n$ for all sufficiently large $n$.
Similarly, we write $u_n = O(v_n)$ if there exists a positive constant $c$ such that $u_n \le c v_n$ for all sufficiently large $n$.
If a sequence of real numbers $\seq{u}{n}{0}$ satisfies a recurrence relation of the form $u_{n+k}=c_{k-1} u_{n+k-1} + \cdots + c_0 u_{n}$ for some real coefficients $c_0,\ldots,c_{k-1}$ and having characteristic polynomial $\chi(x)=x^k - c_{k-1} x^{k-1} - \cdots - c_1 X - c_0$, then $u_n = \sum_{i=1}^r P_i(n) \alpha_i^n$ where
\begin{enumerate}
\item[1)] the real numbers $\alpha_1,\ldots,\alpha_r$ are the distinct roots of $\chi(x)$ with respective multiplicity $m_1,\ldots,m_r$;
\item[2)] for each $i\in\{1,\ldots,r\}$, $P_i$ is a polynomial of degree (strictly) less than $m_i$ whose coefficients are determined by the initial conditions satisfied by $\seq{u}{n}{0}$.
\end{enumerate}
See the general reference~\cite[Chapter~7]{Graham-Knuth-Patashnik-1994}.

\subsection{Dirichlet series, abscissa of convergence, and restrictions}
\label{sec:abscissa}
Given a sequence $\seq{s}{n}{1}$ of real numbers, its associated 
\emph{Dirichlet (generating)  series} is defined by
\[
F_s(z) = \sum_{n\ge 1} \frac{s_n}{n^z}, \ 
\text{where $z \in {\mathbb C}$}.
\]
It is well known that either there exists a real number $r_0$ such 
that $F_s(z)$ converges for all $z$ with $\mathfrak{Re}(z)>r_0$ 
and diverges for all $z$ with $\mathfrak{Re}(z)<r_0$, 
or $F_s(z)$ converges everywhere, or else $F_s(z)$ converges 
nowhere. The real number $r_0$ is called the \emph{abscissa of 
convergence} of the series $F_s(z)$. If $F_s(z)$ does not converge 
anywhere (resp., converges everywhere), we set $r_0 = +\infty$ 
(resp., $r_0=-\infty$).
The \emph{abscissa of absolute convergence} of the series $F_s(z)$ is defined in a similar way when absolute convergence of $F_s(z)$ is considered instead.

\begin{example}
    The most famous example of Dirichlet series is when all 
    coefficients $s_n = 1$ for all $n\ge 1$: the series is the {\em Riemann zeta 
    function}. Its abscissas of convergence and of absolute
    convergence are both equal to $+1$. Another well-studied
    example is when $s_n = (-1)^n$ for all $n\ge 1$: the abscissa of convergence
    of the Dirichlet series $\sum_{n\ge 1} (- 1)^n / n^z$ is equal to 
    $0$, while its abscissa of absolute convergence is $+1$ 
    (this is well known, see, e.g., \cite[p.~10]{Mandelbrojt}).
\end{example}

%\subsection{Restricted Dirichlet series}

In what follows, we consider Dirichlet series in which the 
summation range is restricted or limited to some integers 
sharing a similar property. 
Let $b\ge 2$ be an integer and let $L$ be a language over the alphabet $[0,b-1]$, i.e., a set $L$ of words 
on $[0,b-1]$. We define the \emph{restricted} 
Dirichlet series $F_L(z)$ by
 \[
F_L(z) = \sum_{\substack{n \geq 1 \\ \rep_b(n) \in L}} \frac{1}{n^z},
\]
i.e., the series $F_s(z)$ where, for all $n\ge 1$, the coefficient $s_n$ is 
$1$ or $0$ depending on whether $\rep_b(n)$ belongs to $L$ or not.
In the sequel, the sequence $(s_n)_{n\ge 0}$ is called the \emph{characteristic sequence} of the language $L$ in base $b$.

\medskip

At the beginning of the 20th century, Kempner~\cite{Kempner1914} considered the set of positive integers whose base-$10$ representation contains no $9$.

\begin{theorem}[{Kempner, \cite{Kempner1914}}]\label{th:Kempner}
%Define the alphabet $D=\{0,1,\ldots,8\}$ and $D^*$ the set of words on  $D$.
Consider the language $L_K=[0,8]^*$ over the alphabet $[0,9]$.
Then the restricted harmonic series $F_{L_K}(1)$ converges. 
\end{theorem}

\begin{remark}
This result has now become classical, but it might seem strange since the
``whole'' series $\sum_{n\ge 1} 1/n$ diverges. Actually it is not 
hard to get convinced that ``many'' integers are omitted in the
restricted series (think, for example, of all the integers belonging to the 
interval $[9 \cdot (10)^k, (10)^{k+1}-1]$ for a fixed $k$). 
\end{remark}

Of course Kempner's result implies that the abscissa of convergence 
of the restricted Dirichlet series $F_{L_K}(z)$ is less than or equal 
to $1$. It is then natural to try to determine the (exact) abscissa 
of convergence of this series and similar ones.

\section{A shorter proof of Nathanson's result using K\"ohler and Spilker's method}
\label{sec:short-proof}

In their 2009 paper, K\"ohler and Spilker studied the abscissa of 
convergence of some restricted Dirichlet series.
For this particular restriction, they considered base-$b$ representations that 
avoid some specific set of digits.
%The following result on an abscissa of convergence for a series as above is given in \cite{Kohler-Spilker-2009}.

\begin{theorem}[{K\"ohler and Spilker, \cite[Satz 2]{Kohler-Spilker-2009}}]
\label{thm:German-guys}
Let $b\ge 2$ be an integer.
Consider a non-empty subset $D$ of $[0,b-1]$ with $D\neq \{0\}$.
Then the abscissa of convergence of the restricted Dirichlet series $F_{D^*}(z)$ is equal to $\frac{\log |D|}{\log b}$.
\end{theorem}

In particular, when the set $D$ in~\cref{thm:German-guys} contains all but one element, we obtain a refinement of~\cref{th:Kempner}.
\begin{corollary}
\label{cor:avoid only one letter}
Let $b\ge 2$ be an integer.
Consider the set $D=[0,b-1]\setminus\{a\}$ where $a\in [1,b-1]$.
Then the abscissa of convergence of the restricted Dirichlet series $F_{D^*}(z)$ is equal to $\log_b (b-1)$.
\end{corollary}

In 2021, Nathanson independently considered a different kind of restriction on 
Dirichlet series~\cite{Nathanson2021}.
For an integer base $b\ge 2$, the digits in a given position in the 
base-$b$ representations are constrained according to certain rules.
Depending on the position, the rules may vary; this is a fundamental 
difference between  K\"ohler and Spilker's work and Nathanson's.

\begin{theorem}[{Nathanson, \cite[Theorem 1]{Nathanson2021}}]
\label{thm:Nathanson}
Let $b\ge 2$ be an integer. For all $i\ge 0$, consider a proper 
subset $D_i$ of $[0,b-1]$ and let $\mathcal{D}$ denote the 
sequence $\seq{D}{i}{0}$. 
Define the language
\[
L_{\rm forb(\mathcal{D})}
=
\{w_n w_{n-1}\cdots w_0 : \ w_n \neq 0 \text{ and } w_i \notin D_i, \; 
\forall \, i\in[0,n]
\}.
\]
Suppose furthermore that the set 
${\cal M} = \{m \geq 1 : \ D_{m-1} \neq [1, b-1]\}$\
is infinite and that there exist non-negative real numbers 
$\alpha_0, \alpha_1, \dots, \alpha_{b-1}$ such that for all 
$m \geq 1$ and for all $\ell \in [0, b-1]$, the limit
\[
\lim_{m \to +\infty}\frac{1}{m} | \{i \in [0, m-1] \ : \ |D_i| = \ell \}|
\]
exists and is equal to $\alpha_\ell$.
Then the restricted Dirichlet series $F_{L_{\rm forb(\mathcal{D})}}(z)$
%\[
%F_{L_\mathcal{D}}(z) = \sum_{n\in %\val_b(L_\mathcal{D})} \frac{1}{n^z}
%\]
has abscissa of convergence
\[
\Theta_{\cal D} =\frac{1}{\log b} \sum_{\ell=0}^{b-1} \alpha_\ell \log (b - \ell).
\]
\end{theorem}

Observe that~\cref{thm:Nathanson} clearly 
implies~\cref{thm:German-guys}. However, we 
give a shorter proof of~\cref{thm:Nathanson} using the proof and 
methods from~\cref{thm:German-guys}. As K\"ohler and Spilker did, we require 
\cite[Lemma 1]{Kohler-Spilker-2009}, which is (partially) recalled below and for which K\"ohler and Spilker give a 
standard reference, namely \cite[p.~292]{Titchmarsh-1939}.

\begin{lemma}
\label{lem:Titchmarsh}
Let $\seq{s}{n}{1}$ be a sequence.
Assume that the series 
$A(n)= \sum_{i=1}^n s_i$ of partial sums is positive and tends to infinity. Then the abscissa of 
convergence of the Dirichlet series $F_s(z)$ is
\[
\limsup_{n\to +\infty} \frac{\log A(n) }{\log n}.
\]
\end{lemma}

\begin{proof}[of~\cref{thm:Nathanson}]
First of all, the hypothesis means that, for all $\ell\in[0,b-1]$, 
\begin{equation}
\label{eq: description size}    
    |\{i \in [0, m-1] : |D_i| = \ell\}|
= m \alpha_\ell + o_\ell(m),
\end{equation}
where
\begin{equation}
\label{eq: lim epsilon}
\lim_{m\to +\infty} \frac{o_\ell(m)}{m} = 0. 
\end{equation}

As already mentioned, our strategy to prove the statement is 
inspired by the proof of~\cref{thm:German-guys} and makes 
use of~\cref{lem:Titchmarsh}.
To that aim, we define the sequence $\seq{s}{n}{1}$ to be the characteristic sequence of the language $L_{\rm forb(\mathcal{D})}$ in base $b$
%\[
%s_n
%=
%\begin{cases}
%1, & \text{if } \rep_b(n) \in L_{\rm forb(\mathcal{D})}; \\
%0, & \text{otherwise},
%\end{cases}
%\]
and the sequence $(A(n))_{n\ge 1}$ by $A(n) = |\{ m\in [1,n] : \rep_b(m) \in L_{\rm forb(\mathcal{D})}\}|$ for all $n\ge 1$.
Note that the $n$th partial sum of the sequence $\seq{s}{n}{1}$ is 
$A(n)$. Therefore, using the notation from~\cref{lem:Titchmarsh}, 
we have $F_s(z) = F_{L_{\rm forb(\mathcal{D})}}(z)$.
To prove that the abscissa of convergence of the restricted Dirichlet 
series $F_{L_{\rm forb(\mathcal{D})}}(z)$ is equal to $\Theta_{\cal D}$, it suffices, by~\cref{lem:Titchmarsh}, to prove that
\begin{equation}
\label{eq:limsup-theta}
    \limsup_{n \to +\infty} \frac{\log A(n)}{\log n} = \Theta_{\cal D}.
\end{equation}
To evaluate the quantity $A(n)$ for large enough $n$,
we define $k$ by $b^{k-1} \leq n <  b^k$, i.e., 
$k= \lfloor \frac{\log n}{\log b} \rfloor + 1$. 
Of course, we have 
\begin{equation}
\label{eq: inequalities for A}
A(b^{k-1}) \leq A(n) \leq A(b^k). 
\end{equation}
We divide the rest of the proof into two steps.
First, we prove that the $\limsup$ in~\eqref{eq:limsup-theta} is bounded above by $\Theta_{\cal D}$.
Second, we build up an increasing sequence of integers such that the corresponding $\limsup$ is bounded below by $\Theta_{\cal D}$.

\bigskip

\textbf{Step 1}:
First, we prove that the $\limsup$ in~\eqref{eq:limsup-theta} is at most $\Theta_{\cal D}$.
%%%%Old proof
%We note that $\rep_b(b^\ell)=10^\ell$ for all $\ell\ge 0$.
%So, up to padding representations with leading zeroes, we can consider that all integers in the interval $[1,b^k-1]$ can be represented by words of length $k$ on $[0,b-1]$.
%Hence $A(b^k-1)$ can be bounded from above by the number of words of length $k$ on $[0,b-1]$ such that their $i$-th digits belong to $[0,b-1] \setminus D_i$, i.e.,
%\[
%A(b^k-1) \leq \left(\prod_{i=0}^{k-1} (b-|D_i|)\right) - 1.
%\]
%Note that we subtracted $1$ because the word $0^k$ does not represent an integer in $[1, b^k]$ (actually, it doesn't represent any integer, since  the number $0$ is represented by the empty word).
%Thus
%\[
%A(b^k) \leq A(b^k-1)+1 \leq \prod_{i=0}^{k-1} (b-|D_i|).
%\]
%By~\eqref{eq: description size}, we obtain
%\begin{equation}
%\label{eq: maj on A(bk)}
%    A(b^k) \leq
%\prod_{i=0}^{k-1} (b-|D_i|) = 
%\prod_{\ell=0}^{b-1} (b-\ell)^{\alpha_{\ell} k + o_{\ell}(k)}
%\leq b^{\sum_{\ell = 0}^{b-1} |o_{\ell}(k)|}
%\prod_{\ell=0}^{b-1} (b-\ell)^{\alpha_{\ell} k}.
%\end{equation}
%Putting~\eqref{eq: inequalities for A} and~\eqref{eq: maj on A(bk)} together gives
%\[
%\log A(n) \leq
%(k \log b) \left(\sum_{\ell = 0}^{b-1} \frac{|o_{\ell}(k)|}{k} 
%+ \frac{1}{\log b} \sum_{\ell=0}^{b-1} \alpha_{\ell} 
%\log(b-\ell)\right).
%\]
We note that $\rep_b(b^j)=10^j$ for all $j\ge 0$.
For $j\ge 1$, all integers in the interval $[b^{j-1},b^j-1]$ are represented by words of length $j$ over the alphabet $[0,b-1]$.
Hence $A(b^j-1)-A(b^{j-1}-1)$ can be bounded from above by the number of words of length $j$ 
on $[0,b-1]$ such that their $i$-th digits belong to $[0,b-1] \setminus D_i$, i.e.,
\[
A(b^j-1)-A(b^{j-1}-1) \leq \prod_{i=0}^{j-1} (b-|D_i|).
\]
Thus, if we set $A(0)=0$, we obtain 
\begin{align*}
A(b^k) 
&\leq A(b^k-1)+1 
\leq 1 + \sum_{j=1}^{k} (A(b^j-1)-A(b^{j-1}-1)) \\
&\leq 1 + \sum_{j=1}^{k} \prod_{i=0}^{j-1} (b-|D_i|) \leq 1 + k \prod_{i=0}^{k-1} (b-|D_i|).
\end{align*}
By~\eqref{eq: description size}, we obtain
\begin{equation}
\label{eq: maj on A(bk)}
A(b^k) 
\leq 1 + k \prod_{\ell=0}^{b-1} (b-\ell)^{\alpha_{\ell} k + o_{\ell}(k)}
\leq 1 + k b^{\sum_{\ell = 0}^{b-1} |o_{\ell}(k)|}
\prod_{\ell=0}^{b-1} (b-\ell)^{\alpha_{\ell} k}.
\end{equation}
Putting~\eqref{eq: inequalities for A} and~\eqref{eq: maj on A(bk)} together gives
\[
\log A(n)
\leq
1 + \log k + (k \log b) \left(\sum_{\ell = 0}^{b-1} \frac{|o_{\ell}(k)|}{k} 
+ \frac{1}{\log b} \sum_{\ell=0}^{b-1} \alpha_{\ell} 
\log(b-\ell)\right).
\]
Since $k = \lfloor \frac{\log n}{\log b}\rfloor + 1$ implies that 
$k \log b \sim \log n$ when $n$ goes to infinity, we get
\[
\limsup_{n \to +\infty} \frac{\log A(n)}{\log n} \leq 
\lim_{k \to \infty} \left( \sum_{\ell = 0}^{b-1} 
\frac{|o_{\ell}(k)|}{k}
+ \frac{1}{\log b} \sum_{\ell=0}^{b-1} \alpha_{\ell}\log (b-\ell) \right)
= \Theta_{\cal D},
%= \frac{1}{\log b} \sum_{\ell=0}^{b-1} \alpha_{\ell}\log(b-\ell),
\]
where the last equality follows using~\eqref{eq: lim epsilon}.
This ends the first step. \hfill $\blacksquare$

\bigskip

\textbf{Step 2}: Now we prove that there exists an increasing sequence 
$\seq{m}{t}{0}$ of integers such that 
$\limsup_{m_t \to +\infty} 
\frac{\log A(m_t)}{\log m_t} \geq \Theta_{\cal D}$.
In fact, we will use the sequence $\seq{m}{t}{0}$ consisting of the elements of
${\cal M}$ in increasing order (recall that ${\cal M}$ is defined as in~\cref{thm:Nathanson} above and that ${\cal M}$ is infinite).
In short, for all $t\ge 0$, $m_t$ is the $(t+1)$st element of ${\cal M}$.

\medskip

First, for each $t\ge 0$, we consider the set of integers in $[b^{m_t-1},b^{m_t}-1]$, whose base-$b$ representations have length (exactly) $m_t$ and are such that their $i$th digit is not in $D_i$, i.e., the set 
%$\{n\ge 1 : n\in [b^{m_{t-1}},b^{m_t}-1] \text{ and } \rep_b(n) \in L_{\rm forb(\mathcal{D})} \}$.
\[
\left\{n\ge 1 : n\in [b^{m_t-1},b^{m_t}-1] \text{ and } \rep_b(n) \in L_{\rm forb(\mathcal{D})} \right\}.
\]
We let $P_{m_t}$ denote the cardinality of this set.
Observe that $A(b^{m_t}) \geq P_{m_t}$, so in order to bound the $\limsup$ above, we will evaluate $P_{m_t}$ in the sequel.
Since base-$b$ representations do not start with $0$, we have
\begin{align}
\label{eq:description Pmt}
    P_{m_t} =
\begin{cases}
 \displaystyle \prod_{i=0}^{m_t-1} (b - |D_i|), 
 &\text{if $0 \in D_{m_t-1}$}; \\
 \displaystyle\left(\frac{b - |D_{m_t-1}| -1}{b - |D_{m_t-1}|}\right)
 \prod_{i=0}^{m_t-1} (b - |D_i|), &\text{if $0 \notin D_{m_t-1}$}.
\end{cases}
\end{align}
(Note that, if $0 \notin D_{m_t-1}$, then there are $b - |D_{m_t-1}|-1$ possible choices for the digit in position $m_t-1$, as leading zeroes are forbidden.)
Following \cite{Nathanson2021}, we distinguish two cases to evaluate 
$P_{m_t}$, depending on the cardinality $|D_{m_t - 1}|$.

    \smallskip

\noindent \textbf{Case 1}: If $0 \leq |D_{m_t - 1}| \leq b - 2$, then
    \[
    \frac{b - |D_{m_t-1}| -1}{b - |D_{m_t-1}|} \geq \frac{1}{2}, 
    \]
    and thus
    \[
    P_{m_t} \geq \frac{1}{2} \prod_{i=0}^{m_t-1} (b - |D_i|)
    = \frac{1}{2} 
    \prod_{\ell=0}^{b-1} (b - \ell)^{\alpha_\ell m_t + o_\ell(m_t)}.
    \]
    where we used~\eqref{eq: description size} in the last equality.

    \smallskip
    
\noindent \textbf{Case 2}: If $|D_{m_t - 1}| = b - 1$, then $0 \in D_{m_t-1}$ 
    (recall that $m_t \in {\cal M}$). 
    %Thus, using the equality $b -|D_{m_t-1}| = 1$, we have
Thus, using the first equality in~\cref{eq:description Pmt}, we have
    \[
    P_{m_t} = \prod_{i=0}^{m_t-1} (b - |D_i|) =  
    \prod_{\ell=0}^{b-1} (b - \ell)^{\alpha_\ell m_t + o_\ell(m_t)}
    \]
by using~\eqref{eq: description size} again.

\smallskip

\noindent Hence, in both cases, we obtain
\[
P_{m_t} \geq \frac{1}{2} 
\prod_{\ell=0}^{b-1} (b - \ell)^{\alpha_\ell m_t + o_\ell(m_t)}
\geq \frac{1}{2} \ b^{-\sum_{\ell=0}^{b-1} |o_\ell(m_t)|} 
\prod_{\ell=0}^{b-1} (b-\ell)^{\alpha_\ell m_t},
\]
since (distinguishing the cases $o_\ell(m_t)\ge 0$ and $o_\ell(m_t)<0$)
\[
o_\ell(m_t) \log(b-\ell) \ge -|o_\ell(m_t)| \log b.
\]
This in turn yields
\[
\log A(b^{m_t}) \geq \log P_{m_t} \geq - \log 2 - \log b 
\sum_{\ell=0}^{b-1} |o_\ell(m_t)| +
m_t \sum_{\ell=0}^{b-1} \alpha_\ell \log (b - \ell),
\]
 and thus
 \[
\frac{\log A(b^{m_t})}{\log b^{m_t}} \geq \frac{\log P_{m_t}}{m_t \log b}
\geq - \frac{\log 2}{m_t \log b} - 
\sum_{\ell=0}^{b-1} \frac{|o_\ell(m_t)|}{m_t} +
\frac{1}{\log b} \sum_{\ell=0}^{b-1} \alpha_\ell \log (b - \ell).
\]
Finally, by using~\eqref{eq: lim epsilon}, we obtain
\[
\limsup_{m_t \to +\infty} \frac{\log A(b^{m_t})}{\log b^{m_t}} \geq
\frac{1}{\log b} \sum_{\ell=0}^{b-1} \alpha_\ell \log (b - \ell) = \Theta_{\cal D}.
\]
This finishes up the two steps of the proof. 
The proof is now complete.
\qed
\end{proof}

\section{Opening the door to a myriad of cases}\label{sec:myriad}

In this section, we illustrate with various practical examples how K\"ohler and Spilker's strategy using 
\cref{lem:Titchmarsh} may open the way to generalizations 
of~\cref{thm:German-guys,thm:Nathanson} (note that none of these results may be applied to the cases considered in~\cref{sec:myriad}).
We focus on words that periodically avoid certain blocks and the key is to \emph{count} them. Counting words 
having a prescribed structure is notably done in several papers of 
Janji\'c 
\cite{Janjic2015,Janjic2016,Janjic2017,Janjic2018-2,Janjic2018-1}.
An important paper on the subject is, of course, the 1999
Noonan and Zeilberger paper \cite{Noonan-Zeilberger-1999}, that 
describes and implements the ``Goulden-Jackson cluster method''.

\medskip

We start with a lemma that will prove useful.

\begin{lemma}\label{le:useful}
Let $\seq{s}{n}{1}$ be a sequence of non-negative numbers 
and let $(A(n))_{n \geq 1}$ be its summatory function, i.e., $A(n) = \sum_{i=1 \leq i \leq n} s_i$. 
Suppose that there exist an integer $b \ge 2$ and a real 
$\lambda > 1$ such that $A(b^k) = \lambda^{k(1+o(1))}$ 
when $k$ goes to infinity. 
Then, $\frac{\log A(n)}{\log n} \sim \frac{\log \lambda}{\log b}$
when $n$ goes to infinity.
\end{lemma}

Note that the assumption in the previous lemma is equivalent to saying 
that $\log A(b^k) = k(1+o(1)) \log \lambda$, i.e., 
$\log A(b^k) \sim k \log \lambda$.

\begin{proof}[of~\cref{le:useful}]
    For an integer $n\geq 1$, let $k$ be such that $b^{k-1} \leq n < b^k$
    (i.e., $k-1 = \lfloor \frac{\log n}{\log b} \rfloor$). 
    We have $k \sim \frac{\log n}{\log b}$ when $n$ goes to infinity.
    Since $A(b^{k-1}) \leq A(n) \leq A(b^k)$, we have
    \[
    \frac{\log A(b^{k-1})}{\log n} \leq \frac{\log A(n)}{\log n} 
    \leq \frac{\log A(b^k)}{\log n}
    \]
    for $n\ge 2$.
    By assumption both $\log A(b^k)$ and $\log A(b^{k-1})$ behave like $k \log \lambda$, thus 
    \[\displaystyle \lim_{n\to +\infty} \frac{\log A(b^{k-1})}{\log n} =
    \frac{\log \lambda}{\log b} \ \ \text{and} \ \ 
    \displaystyle \lim_{n\to +\infty}\frac{\log A(b^k)}{\log n} =
    \frac{\log \lambda}{\log b} .
    \] 
    Hence, the squeeze theorem finally gives $\lim_{n\to +\infty} \frac{\log A(n)}{\log n} =
    \frac{\log \lambda}{\log b}$, as desired. 
    \qed
\end{proof}

\subsection{All-distinct-letter blocks}
\label{sec:disjoint blocks}

We play with a specific example in base $10$ where representations 
periodically avoid two blocks that do not share a letter.
More specifically, define $L_1$ to be the language of all 
the words $w_k w_{k-1} \cdots w_0$ over $[0, 9]$ such that $w_k \neq 0$
and, for all $i\in[0,k-1]$, 
\[
w_{i+1}w_i \neq 
\begin{cases} 
12, & \text{if $i$ is even};\\
89, & \text{if $i$ is odd};
\end{cases}
\]
%and define $L'_1$ to be the same language without the condition $w_k\neq 0$. 
Observe that $L_1$ contains base-$10$ representations that periodically avoid the blocks $12$ and $89$.
Note that the blocks $12$ and $89$ do not have a letter in common.
%Also note that words in $L_1$ may not contain the block $12$ (resp., $89$) at all; but if they do, then it must be at an odd (resp., even) position.
Also note that if words in $L_1$ contain the block $12$ (resp., $89$), then it must occur at an odd (resp., even) position; it can also be the case that they do not contain either block.

The goal of this section is to determine the abscissa of convergence 
of the restricted Dirichlet series $F_{L_1}(z)$.
To this aim, for all $n\ge 0$, we let $v_n$ denote the cardinality of the sub-language 
$L_1 \cap [0,9]^n$, i.e., $v_n$ gives the number of length-$n$ words in $L_1$. 
One can also compute the first few elements of the sequence 
$(v_n)_{n\ge 0}$ by hand and obtain $1, 9, 89, 881, 8721$.
Entering them in Sloane's On-Line Encyclopedia of Integer Sequences (OEIS), we find the related sequences \cite[A072256]{Sloane} and
\cite[A138288]{Sloane}.
In the description of the first, we read that the $n$th term gives 
the number of ``$01$-avoiding words of length $n$ over the alphabet $[0,9]$ which do not end in $0$''.
Quite clearly, by reversing the reading order, it also gives the number of length-$n$ words over $[0,9]$ not starting with $0$ and avoiding the pattern $10$.
If we denote the latter sequence by $\seq{x}{n}{0}$, Sloane's OEIS provides a recurrence relation given by $x_0 = 1$, $x_1=9$, and, for all $n\ge 0$,
\begin{equation}
\label{eq: recurrence for 2 non-overlapping blocks}
    x_{n+2}=10 x_{n+1} - x_{n}. 
\end{equation}

We now establish a bijection between the language $L_1$ and that containing words satisfying the condition described above.
We start off with the following result.

\begin{proposition}
\label{pro:bijection-blocks}
    Let $X_1$ be the set of words over the alphabet $[0,9]$ containing at least one of the blocks $12,89$ and such that every block $12$ (resp., $89$) necessarily occurs at even (resp., odd) positions.
    Let $X_2$ be the set of words over the alphabet $[0,9]$ containing the block $12$.
    %Define the map $f: [0,9]^*\to [0,9]^*$ that sends any word to a word obtained by replacing in the initial word all occurrences of the block $89$ in odd positions by the block $12$.
    Define the map $f: [0,9]^*\to [0,9]^*$ that sends a word to that obtained by replacing all occurrences of the blocks $89$ (resp., $12$) in odd positions by the block $12$ (resp., $89$).
    Then the map $f_1: X_1 \to X_2, u \mapsto f(u)$ is a bijection.
\end{proposition}
\begin{proof}
    First, note that for all $u\in X_1$ we have $f_1(u)\in X_2$.
    Indeed, if $u$ contains a block $12$ occurring at an even position, then $12$ is clearly a factor of $f_1(u)$, so $f_1(u)\in X_2$.
    Otherwise, $u$ necessarily contains a block $89$ at an odd position and similarly, $f_1(u)\in X_2$.

    Now consider the map $f_2: X_2 \to X_1, v \mapsto f(v)$.
    Proving that $f_2(X_2) \subseteq X_1$ follows the same lines as above.
    Now observe that we have $f_2(f_1(u))=u$ for all $u\in X_1$ and $f_1(f_2(v))=v$ for all $v\in X_2$.
    %
    %%Old and wrong proof
    %First of all, the map $f: X_1\to X_2$ is well-defined.
    %
    %We first show that $f$ is a surjection.
    %Let $v\in X_2$.
    %We build up a word $u\in X_1$ such that $f(u)=v$.
    %If $v$ does not contain any occurrence of the block $12$ in odd positions, then setting $u=v$ is enough.
    %Otherwise, we define $u$ by replacing in $v$ each occurrence of the block $12$ in odd positions by the block $89$.
    %
    %We next show that $f$ is an injection.
    %Take $w_1,w_2\in X_1$ with $f(w_1)=f(w_2)$.
    %We show that $w_1=w_2$.
    %We proceed by contradiction and we assume that $w_1\neq w_2$.
    %Since $w_1,w_2$ have equal images under $f$, the words $w_1,w_2$ differ by an occurrence of $89$ (recall that blocks $89$ only appear at odd positions).
    %Starting from the right, there is a ``smallest'' position in which the two words differ: we may write $w_1=x_1 89 y_1$ and $w_2=x_2 89 y_2$ with odd $|y_1|,|y_2|$ and $|y_1|\neq|y_2|$.
    %Without loss of generality, we may even assume that $|y_1|>|y_2|$.
    %As we identify the smallest position, we even have that $89y_2$ is a suffix of $y_1$.
    %So $f(y_1)$ can be written as $z 12 f(y_2)$.
    %Since $y_1\in X_1$ and $12$ occurs in $f(y_1)$ in the odd position $|y_2|$, the definition of $f$ implies that $y_1$ contains an occurrence of $89$ in position $|y_2|$.
    %This contradicts the minimality of the occurrences of the blocks we have identifies in $w_1,w_2$. 
    \qed
\end{proof}

Using~\cref{pro:bijection-blocks}, complementing the languages, and getting rid of words starting with $0$, we link our sequence of interest and that of Sloane's.

\begin{corollary}
\label{cor:link-Ln-seq-OEIS-block10}
    For all $n\ge 0$, we have $v_n=x_n$.
\end{corollary}

\cref{cor:link-Ln-seq-OEIS-block10} and the work in~\cite[\S 3]{Kohler-Spilker-2009} imply the next result.
Indeed, in that section, K\"ohler and Spilker considered the language of base-$b$  representations avoiding a fixed block of digits and determine the abscissa of convergence of the corresponding restricted Dirichlet series.

\begin{theorem}
\label{thm:abscissa-convergence-blocks-12-89}
    Consider the dominant root $\lambda=5+2 \sqrt{6}$ of the characteristic polynomial $P(x)=x^2-10x+1$ of the recurrence in~\eqref{eq: recurrence for 2 non-overlapping blocks}.
    %other root: and $5-2 \sqrt{6}$.
    Then the abscissa of convergence of the restricted Dirichlet series $F_{L_1}(z)$ is $\frac{\log\lambda}{\log 10}$.
\end{theorem}
%\begin{proof}
%    Observe that the sequence $(|L_n|)_{n\ge 1}$ 
%    has the same asymptotic growth as the sequence $\seq{x}{n}{0}$, i.e., $|L_n| \sim \lambda^n$.
%    So from~\cref{lem:Titchmarsh}, we deduce that the abscissa of convergence of the restricted Dirichlet series $F_L(z)$ is $\lim_{n \to +\infty} \frac{\log |L_n|}{\log n} = \lambda$. \qed
%\end{proof}

An alternative approach to computing the number $v_n$ of length-$n$ words in $L_1$ uses the Goulden-Jackson cluster method; see, for example,
\cite{Noonan-Zeilberger-1999}.   This general technique provides an algorithm to determine the generating function for the number of length-$n$ words over a finite alphabet that avoid some given finite set of patterns.   At first glance $L_1$ is not defined in terms of a finite list of
avoided patterns (since it depends on the parity of the position), but we can nevertheless resort to the following trick:  expand the alphabet from $\{0,1, \ldots, 9 \}$ to
an alphabet of twice the size, containing both primed and unprimed digits.   Define $d_n$ to be the number of length-$n$ words over this larger alphabet that avoid $ab$ and $a'b'$ for $a,b\in[0,9]$ and also $1'2$ and $89'$.
Then the Goulden-Jackson cluster method, as implemented in Maple by the {\tt DAVID\_IAN} package, demonstrates that  $\sum_{n \geq 0} d_n x^n = (1 +10x-x^2)/(1-10x+x^2)$. 
Here $d_n$ counts the number of words that avoid 12 at even positions and 89 at odd positions (if the last digit is primed), plus the number of words that avoid 12 at odd positions and 89 at even positions (if the last digit is not primed). The two summands are clearly in bijection with each other, and it now follows that $(d_{n+1}-d_n)/2$ is equal to $v_{n+1}$.

\subsection{Blocks sharing letters}
\label{sec:blocks sharing letters}

We now turn to the case where base-$10$ representations  periodically avoid two blocks sharing common letters. It is a little 
trickier to handle and is \emph{a priori} neither a consequence of K\"ohler and 
Spilker's results~\cite{Kohler-Spilker-2009} nor 
Nathanson's~\cite{Nathanson2021}.
%However the presentation of the current section follows the same lines as the previous one.
As in the previous section, we define $L_2$ to be the language of all words $w_k w_{k-1} \cdots w_0$ over $[0,9]$ such that $w_k \neq 0$ and, for all $i\in[0,k-1]$, 
\[
w_{i+1}w_i \neq 
\begin{cases} 
12, & \text{if $i$ is even};\\
21, & \text{if $i$ is odd};
\end{cases}
\]
and also define $L_2'$ to be the same language without the condition $w_k\neq 0$.
Observe that $L_2$ contains base-$10$ representations that periodically avoid the blocks $12$ and $21$.

Similarly, we want to determine the abscissa of convergence of the restricted Dirichlet series $F_{L_2}(z)$.
So, for all $n\ge 0$, we let $v_n$ denote the cardinality of the sub-language $L_2  \cap [0,9]^n$ and, for all $n\ge 1$, we also define the sets
\begin{align*}
    S_n &= \{ w_{n-1}\cdots w_0 \in L_2 : w_{n-1} \neq 2 \text{ if $n$ is even and } w_{n-1} \neq 1 \text{ otherwise}\}, \\
    P_n &= \{ w_{n-1}\cdots w_0 \in L_2 : w_{n-1} = 2 \text{ if $n$ is even and } w_{n-1} = 1 \text{ otherwise}\}, \\
    Q_n &= \{ w_{n-1}\cdots w_0 \in L_2' : w_{n-1} = 0\}.
\end{align*}
We have $|S_1| = 8$, $|P_1|=1=|Q_1|$, and 
\begin{align*}
    |S_{n+1}| &= 8 (|S_n| + |P_n| + |Q_n|), \\
    |P_{n+1}| &= |S_n| + |Q_n|,\\
    |Q_{n+1}| &= |S_n| + |P_n| + |Q_n|,
\end{align*}
for all $n \ge 1$.
For example, let us explain the first equality.
To build up a word in $S_{n+1}$, we may prepend a letter $a\in[0,9]$ to a word $w$ of length $n$, according to one of the next three cases: if $w\in S_n$ (resp., $P_n$; resp., $Q_n$), then $a\in[0,9]$ except $0$, and either $1$ or $2$, depending on the parity; this gives $8$ possibilities.
%(Note that the previous recurrence relations are different from those in~\eqref{eqn:recurrence-rel-for-12-89-1}.)

As we have $(v_n)_{n\ge 1}=(|S_n|+|P_n|)_{n\ge 1}$, the first few values of $(v_n)_{n\ge 0}$ are
\[
1, 9, 89, 882, 8739, 86589, 857952, 8500869, 84229389, 834572322. 
\]
In contrast with the previous section, this sequence does not \textit{directly} belong to the OEIS.
A fortunate error made by the authors nevertheless led them to the sequence~\cite[A322054]{Sloane}.
For all $n\ge 0$, its $n$th term gives the number of ``decimal strings of length $n$ that do not contain a specific string $aa$ (where $a$ is a fixed single digit)''.
If we denote the latter by $\seq{x}{n}{0}$, the OEIS provides us with a recurrence relation given by $x_0 = 1$, $x_1 = 10$, and 
\begin{equation}
\label{eq: recurrence for 2 overlapping blocks}
    x_{n+2}=9 (x_{n+1} + x_{n})
\end{equation}
for all $n\ge 0$. 
Its first few terms are
\[
1, 10, 99, 981, 9720, 96309, 954261, 9455130, 93684519.
\]
We will establish that the sequence $(v_n)_{n\ge 0}$ can be obtained as the first difference of the sequence $\seq{x}{n}{0}$.
We start with a result analogous to~\cref{pro:bijection-blocks}.

\begin{proposition}
\label{pro:bijection-blocks-overlapping}
    Let $Y_1$ be the set of words over the alphabet $[0,9]$ containing at least one of the 
 blocks $12,21$ and such that every block $12$ (resp., $21$) necessarily occurs at even (resp., odd) positions.
    Let $Y_2$ be the set of words over the alphabet $[0,9]$ containing the block $aa$, for a fixed letter $a\in[0,9]$.
    %Old: Define the map $f: [0,9]^*\to [0,9]^*$ that sends any word to a word obtained by replacing in the initial word any letter $1$ (resp., $2$) in an odd (resp., even) position by the letter $a$.
    Define the map $f_a: [0,9]^*\to [0,9]^*$ that sends every word to a word obtained by replacing every $1$ (resp., $2$) in an even (resp., odd) position by the letter $a$ and every letter $a$ in an even (resp., odd) position by the letter $1$ (resp., $2$).
    Then the map $f_{a,1}: Y_1\to Y_2, u \mapsto f_a(u)$ is a bijection.
    Furthermore, if the letter $a$ is non-zero, then $v_{n+1}=x_{n+1} - x_n$ for all $n\ge 0$.
\end{proposition}
\begin{proof}
    The proof of the first part of the statement follows the same lines as that of~\cref{pro:bijection-blocks} by considering the map $f_{a,2}: Y_2\to Y_1, v \mapsto f_a(v)$ instead.
    Let us prove the second part.
    The number of decimal length-$(n+1)$ words avoiding $aa$, which is $x_{n+1}$, is equal to the sum of the number of decimal length-$n$ words avoiding $aa$ with the number of decimal length-$(n+1)$ words avoiding $aa$ and not starting with $0$.
    The first number is $x_n$ by definition and the second is $v_{n+1}$ by the first part of the statement.
    Indeed, the first part of the statement implies that the number of length-$n$ words in $Y_1$ is equal to the number of decimal length-$n$ words containing $aa$.
    Now complementing the two languages, getting rid of words starting with $0$ and since $a\neq 0$, we obtain $x_{n+1} = v_{n+1} + x_n$, as desired.
    \qed
\end{proof}

%Using~\cref{pro:bijection-blocks-overlapping}, complementing the languages, and getting rid of words starting with $0$, we link our sequence of interest and that of Sloane's.

%\begin{corollary}
%\label{cor:link-Ln-seq-OEIS-blockaa}
%    Assume that the letter $a$ defining the language $Y_2$ in~\cref{pro:bijection-blocks-overlapping} is non-zero.
%    Then, for all $n\ge 0$, $|L_{n+1}|=x_{n+1} - x_n$.
%\end{corollary}
%\begin{proof}
%    The number of decimal length-$(n+1)$ words avoiding $aa$, which is 
%    $x_{n+1}$, is equal to the sum of -the number of decimal length-$n$ words avoiding $aa$ with the number of decimal length-$(n+1)$ words avoiding $aa$ and not starting with $0$. The first number is is $x_n$ by definition and the second is $|L_{n+1}|$
%    by~\cref{pro:bijection-blocks-overlapping} since $a\neq 0$.
%    So $x_{n+1} = |L_{n+1}| + x_n$, as desired. \qed
%\end{proof}    

\begin{theorem}
\label{thm:abscissa-convergence-blocks-12-21}
    Consider the dominant root $\lambda= \frac{3}{2} (3+\sqrt{13})$ 
    of the characteristic polynomial $P(x)=x^2-9x-9$ of the recurrence in 
    \eqref{eq: recurrence for 2 overlapping blocks}.
    %other root: and $\frac{3}{2} (3-\sqrt{13})$.
    Then the abscissa of convergence of the restricted Dirichlet 
    series $F_{L_2}(z)$ is $\frac{\log \lambda}{\log 10}$.
\end{theorem}
\begin{proof}
    %Let $s_n$ be equal to $1$ if $\rep_{10}(n)\in L_2$, $0$ otherwise,
    Let $(s_n)_{n\ge 0}$ be the characteristic sequence of $L_2$ in base $b=10$ and let $A(n)$ denote the summatory function of the sequence $(s_n)_{n\ge 0}$, i.e., $A(n) = \sum_{1 \leq i \leq n} s_i$.
    By~\cref{lem:Titchmarsh}, the abscissa of convergence of the restricted Dirichlet 
    series $F_{L_2}(z)$ is governed by the asymptotics of $(\log A(n) /\log n)_{n\ge 0}$.
    Observe that the behavior of $A((10)^n)$ is determined by that of $x_n$: indeed, $A((10)^n)$ counts the integers $m$ with $1\le m\le (10)^n$ having a base-$10$ representation in $L_2$, so $v_1 + \cdots + v_n +1=x_n$ by~\cref{pro:bijection-blocks-overlapping}.
    %$v_{n+1}=x_{n+1}-x_n$ by~\cref{pro:bijection-blocks-overlapping}.
    Now note that the sequence $(x_n)_{n\ge 0}$ satisfies $x_n \sim C \lambda^n$ when $n$ goes to infinity for some positive constant $C$ (we even have $C \approx 1.0084$; recall the end of~\cref{sec:notation and def}).
    %Thus $v_n \sim C\lambda^n$ with $C = (\lambda - 1)/\lambda$ by~\cref{pro:bijection-blocks-overlapping}.
    So~\cref{le:useful} implies that the abscissa of convergence of $F_{L_2}(z)$ is $\frac{\log \lambda}{\log 10}$. \qed
\end{proof}

We can also count the number of words in $L_2$ using the Goulden-Jackson cluster method, using the same idea as in~\cref{sec:disjoint blocks}.
Define $d_n$ to be the number of length-$n$ words over the larger alphabet $\{a,b'\mid 0\le a,b\le 9\}$ that avoid $ab$ and $a'b'$ for $a,b\in[0,9]$ and also $1'2$ and $21'$.
Then, the Goulden-Jackson cluster method, as implemented in Maple by the {\tt DAVID\_IAN} package, demonstrates that  $\sum_{n \geq 0} d_n x^n = (1 +11x+9x^2)/(1-9x-9x^2)$.  It now follows that $(d_{n+1}-d_n)/2$ is equal to $v_{n+1}$.

\subsection{\textit{To infinity and beyond!}}

Other results, involving in particular several sequences in the OEIS, can be worked out using similar methods, but also
by using the method in \cite{Noonan-Zeilberger-1999}.
%We plan to address them in the long version of this paper, noting in particular that many sequences in the OEIS occur in this context.
%A sneak preview (that does not use \cite{Noonan-Zeilberger-1999}) is given in this section.
Here, we give examples that do not use \cite{Noonan-Zeilberger-1999}.
For example, inspired by the previous section, we fix two integers 
$b,k\ge 2$ and we define a sequence 
$\seq{y}{n}{0}$ analogous to $\seq{x}{n}{0}$ from~\cref{sec:blocks sharing letters}: we set $y_n=b^n$ 
for all $n\in [0,k-1]$ and, for all $n\ge k$, we define
\begin{equation}
\label{eq: 3rd recurrence}
    y_n = (b-1) \cdot \sum_{i=1}^{k} y_{n-i}.
\end{equation}

\begin{proposition}
\label{pro: patter avoidance and sequence of first differences}
    The sequence $(y_{n+1} - y_n)_{n\ge 0}$ counts the number of 
    length-$(n+1)$ words over $[0,b-1]$ that do not start with $0$ and 
    that avoid the $k$-th power $a^k$, for some fixed letter $a\in [0,b-1]$.
\end{proposition}
\begin{proof}
    For all $m\ge 0$, $y_m$ gives the number of length-$m$ words 
    over $[0,b-1]$ that do not contain the $k$th power $a^k$, so $y_{n+1}-y_n$ gives the expected number. \qed
\end{proof}

\begin{example}
For all $k\ge 2$, the sequence $\seq{y}{n}{0}$ corresponding to $b=2$ 
gives the $k$-bonacci numbers: indeed, binary $k$-bonacci words avoid 
the block $1^k$.
For $k=2$ and $b=10$, we again find the sequence $\seq{x}{n}{0}$ defined by Equation~\eqref{eq: recurrence for 2 overlapping blocks}. 
Other cases are displayed in~\cref{tab:examples of seq in Sloane}.
\begin{table}
    \centering
    \caption{Depending on the values of the parameters $k$ and $b$, the entry number in~\cite{Sloane} of the sequence $\seq{y}{n}{0}$ is given (sometimes, one needs to crop the first few terms).} 
    \label{tab:examples of seq in Sloane}
    \begin{tabular}{c|c|c}
       $k$ & $b$ & entry number of $\seq{y}{n}{0}$ in~\cite{Sloane} \\
       \hline
       $2$ & $3$ & A028859, A155020 \\
       $2$ & $4$ & A125145 \\
       $2$ & $5$ & A086347 \\
       $2$ & $6$ & A180033 \\
       $2$ & $7$ & A180167 \\
       $2$ & $10$ & A322054 \\
       \hline
       $3$ & $3$ & A119826 \\
       $3$ & $4$ & A282310
    \end{tabular}
\end{table}
\end{example}

We study the characteristic polynomial associated with 
the sequence $\seq{y}{n}{0}$.
Recall that a real number is \emph{Pisot} 
if it is an algebraic integer larger than $1$ whose (other) algebraic conjugates have modulus less than $1$.

\begin{proposition}
\label{pro:dominant-root}
The characteristic polynomial $P(x)= x^k- (b-1) \sum_{i=0}^{k-1} x^i$ of the recurrence in~\eqref{eq: 3rd recurrence} has a dominant root that is a Pisot number and belongs to the interval $(1,b)$.
%Furthermore, $P(x)$ has one negative root if $k$ is even (and it belongs to $[-1,0]$), none if $k$ is odd.
\end{proposition}
\begin{proof}
From the proof of~\cite[Theorem~2 (p.~254)]{Brauer-1951}, we deduce that the polynomial $P(x)$ has one real root $\beta$ with modulus larger than $1$ and all other roots with modulus less than $1$.
Next, observe that $P(1) = 1 - k (b-1) < 0$ and $P(b)=b^k- (b-1) \frac{b^k-1}{b-1}=1>0$, so $\beta \in (1,b)$.
All in all, this proves the claim.
%this implies that $P(x)$ admits exactly one positive  real root in the interval $(1,b)$, which is also a Pisot number.
\qed
\end{proof}

Let us come back to our Dirichlet series of interest.

\begin{theorem}
\label{thm: avoiding block of a}
Let $L_3$ be the set of words over $[0,b-1]$ that do not start with 
$0$ and that avoid the $k$-th power $a^k$, for some fixed letter $a\in [0,b-1]$.
Let $\lambda$ be the dominant root of the characteristic polynomial of the recurrence in~\eqref{eq: 3rd recurrence}.
Then the abscissa of convergence of the series $F_{L_3}(z)$ is equal to 
$\frac{\log\lambda}{\log b}$.
\end{theorem}
\begin{proof}[Sketch]
The proof is similar to that of~\cref{thm:abscissa-convergence-blocks-12-21}.
%By~\cref{pro:dominant-root}, the growth rate of the sequence $(y_n)_{n\ge 0}$ is $\lambda^n$, so that of the sequence of first differences $(y_{n+1}-y_n)_{n\ge 1}$ is $C\lambda^n$ with $C=(\lambda-1)/\lambda$.
By~\cref{pro:dominant-root}, the growth rate of the sequence $(y_n)_{n\ge 0}$ is $\lambda^n$.
We conclude by using~\cref{lem:Titchmarsh,le:useful} together with~\cref{pro: patter avoidance and sequence of first differences}. \qed
\end{proof}

\section{Alternative approaches}
\label{sec:alternative approach}

As indicated in~\cref{lem:Titchmarsh}, the abscissa of
convergence of the Dirichlet series $F_s(z)=\sum_{n\ge 1} s_n/n^z$ can be
deduced from the asymptotic behavior of the summatory 
function $A(n)=\sum_{1 \le  i\leq n} s_i$. For the case of a $b$-automatic (or $b$-regular) sequence $(s_n)_{n \geq 0}$, an abundant literature studies this behavior: we 
will restrict ourselves to cite one of the papers of Dumont and Thomas \cite{Dumont-Thomas}, the pioneering work of Dumas~\cite{Dumas-2007,Dumas-2013,Dumas-2014}, the chapter of Drmota and Grabner \cite{Drmota-Grabner} (in 
\cite{CANT2010}), and the paper of Heuberger and Krenn 
\cite{Heuberger-Krenn-2020}. In this section, we will 
take advantage of the fact that the sequences considered so far are automatic. 
Nevertheless, as will be seen in~\cref{sec:unusual} (with~\cref{pro:non-auto}), our previous method can be applied outside the scope of $b$-regular sequences.

Before giving alternative proofs of the main results in~\cref{sec:myriad}, we show how the following powerful result by Drmota and Grabner handles the simple case of avoiding a single letter.

\begin{theorem}[{Drmota and Grabner, \cite[Theorem 9.2.15]{CANT2010}}]
\label{thm: Drmota and Grabner}
    Let $b\ge 2$ be an integer.
    Let $\seq{s}{n}{0}$ be a $b$-regular sequence whose linear representation is given by the vectors $V,W$ and the matrices $M_0,\ldots,M_{b-1}$.
    Write $M_{s,b} = \sum_{i=0}^{b-1} M_i$ and assume that $M_{s,b}$ has a unique eigenvalue $\lambda>0$ of maximal modulus and that $\lambda$ has algebraic multiplicity $1$.
    Assume further that $\lambda > \max_i ||M_i||$ for some matrix norm $||\cdot||$.
    Let $\mu$ be the modulus of the second largest eigenvalue.
    Then there exists a periodic continuous function $\Phi$ such that 
    \[
    \sum_{i < n} s_i = n^{\log_b \lambda} \Phi(\log_b n) + O(n^{\log_b \mu}) + O(\log n).
    \]
\end{theorem}

We rephrase~\cref{cor:avoid only one letter}.

\begin{corollary}
\label{cor:avoid only one letter bis}
    Let $b\ge 2$ be an integer and let $a$ be a non-zero letter in $[1,b-1]$.
    Let $L_4$ be the language over $[0,b-1]$ of base-$b$ representations that avoid the letter $a$.
    Then the abscissa of convergence of the series $F_{L_4}(z)$ is equal to $\log_{b} (b-1)$.
\end{corollary}
\begin{proof}[Sketch]
    We use the notation of~\cref{thm: Drmota and Grabner}.
    Let $\seq{s}{n}{0}$ be the characteristic sequence of $L_4$ in base $b$.
    %For all $n\ge 0$, let $s_n = 1$ if $\rep_b(n)\in L_4$, $s_n=0$ otherwise.
    By definition, we have $s_{b n + c} = s_n$ for all $c\neq a$ and $s_{b n + a} = 0$.
    The sequence $(s_n)_{n\ge 0}$ is $b$-regular with $M_c=1$ for all $c\neq a$ and $M_a = 0$.
    Thus $\lambda = b-1$.
    Using~\cref{lem:Titchmarsh,thm: Drmota and Grabner}, the abscissa of convergence of the Dirichlet series $F_{L_4}(z)$ is $\log_{b} (b-1)$. \qed
\end{proof}

We now provide other methods to prove~\cref{thm:abscissa-convergence-blocks-12-89}.

\begin{proof}[Alternative proofs of~\cref{thm:abscissa-convergence-blocks-12-89}]
Consider the language $L_1$ from~\cref{sec:disjoint blocks} and its characteristic sequence $\seq{s}{n}{0}$ in base $10$, 
%by $s_n=1$ if $\rep_{10}(n)\in L_1$, $s_n=0$ otherwise
whose first few terms are 
\[
1, 1, 1, 1, 1, 1, 1, 1, 1, 1, 1, 1, 0, 1, 1, 1, 1, 1, 1, 1, 1.
\]
We have $|L_1 \cap [0,9]^n|=\sum_{i = 10^{n-1}}^{10^n-1} s_i$.
The sequence $\seq{s}{n}{0}$ is $10$-automatic: a deterministic finite automaton with output (DFAO) reading base-$10$ representations (starting with the least significant digit) and outputting the sequence is given in~\cref{fig:automata}.
Note that the computations in this section are done using the \emph{Mathematica} package \textsc{IntegerSequences}~\cite{Rowland18Package-1,Rowland18Package-2} written by Rowland.
\begin{figure}
    \centering
    \includegraphics[scale=0.85]{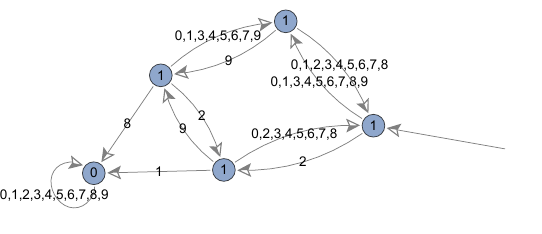}
  %  \quad
  %  \includegraphics[scale=0.75]{automate-12-21.pdf}
    \caption{The sequence $\seq{s}{n}{0}$ is generated by the base-$10$ DFAO (it reads base-$10$ representations with least significant digit first)}.
    \label{fig:automata}
\end{figure}
In fact, the $b$-kernel of a $b$-automatic sequence is in bijection with the set of states in some $b$-automaton that generates it with the property that leading zeroes do not affect the output~\cite[Theorem 6.6.2]{Allouche-Shallit-2003}.
In our case, the $10$-kernel of $\seq{s}{n}{0}$ contains the four sequences
\begin{align*}
    (s_n)_{n\ge 0} &= 1, 1, 1, 1, 1, 1, 1, 1, 1, 1, 1, 1, 0, \ldots,\\
    (s_{10n})_{n\ge 0} &= 1, 1, 1, 1, 1, 1, 1, 1, 1, 1, 1, 1, 1, \ldots,\\
    (s_{10n+2})_{n\ge 0} &= 1, 0, 1, 1, 1, 1, 1, 1, 1, 1, 1, 0, 1,\ldots,\\
    (s_{100n+90})_{n\ge 0} &= 1, 1, 1, 1, 1, 1, 1, 1, 0, 1, 1, 1, 0,\ldots,
\end{align*}
and the zero-sequence $(s_{100n+12})_{n\ge 0} = 0,0,0,\ldots$.
The $4\times 4$ matrix whose columns are built on the terms corresponding to $n\in\{0,1,8,12\}$ in each sequence has a non-zero determinant.
Consequently, the $\mathbb{Q}$-vector space generated by the $10$-kernel of $\seq{s}{n}{0}$ is finitely generated by the four linearly independent sequences $(s_n)_{n\ge 0}$, $(s_{10n})_{n\ge 0}$, $(s_{10n+2})_{n\ge 0}$, and $(s_{100n+90})_{n\ge 0}$.
Using the $10$-automaton from~\cref{fig:automata}, one builds up a linear representation $(V,M_0,\ldots,M_9,W)$ of $\seq{s}{n}{0}$ viewed as a $10$-regular sequence.
%For example, we obtain the relations $s_{100n}=s_{n}$, $s_{100n+20}=s_{n}$, and $s_{1000n+900}=s_{10n}$, so
%\[
%\begin{pmatrix}
%s_{10n}\\
%s_{100n} \\
%s_{100n+20}\\
%s_{1000n+900}
%\end{pmatrix}
%=
%\underbrace{\begin{pmatrix}
%0 & 1 & 0 & 0\\
%1 & 0 & 0 & 0\\
%0 & 1 & 0 & 0\\
%1 & 0 & 0 & 0
%\end{pmatrix}}_{=M_0}
%\begin{pmatrix}
%s_{n}\\
%s_{10n} \\
%s_{10n+2}\\
%s_{100n+90}
%\end{pmatrix}.
%\]
The sum matrix $M_{s,10}=M_0 + \cdots + M_9$ has eigenvalues $\pm(5+2 \sqrt{6})$ and $\pm(5-2 \sqrt{6})$ (each with multiplicity $1$) and does not fulfill the condition of~\cref{thm: Drmota and Grabner}.
Here, we are at a crossroad: either we stick to base $10$ and use a different result, or we change base.

In the first case, Heuberger and Krenn's more general result~\cite[Theorem A]{Heuberger-Krenn-2020} shows that the growth of the main term of $\sum_{i=1}^{n} s_i$ is $n^{\log_{10} \alpha}$ with $\alpha=5+2 \sqrt{6}$.
(Note that we take the liberty not to recall~\cite[Theorem A]{Heuberger-Krenn-2020} in full as its statement is quite lengthy.)
Taking the logarithm and applying~\cref{lem:Titchmarsh} thus gives~\cref{thm:abscissa-convergence-blocks-12-89}.

In the second case, we use the following trick: it is well-known that a sequence is $b$-automatic if and only if it is $b^\ell$-automatic for all $\ell\ge 1$ (see, for instance, \cite[Theorem 6.6.4]{Allouche-Shallit-2003}).
In our case, the sequence $\seq{s}{n}{0}$ is thus $100$-automatic.
One can build a base-$100$ DFAO generating $\seq{s}{n}{0}$ (it has three states) as well as a linear representation $(V',M'_0,\ldots,M'_{99},W')$ of $\seq{s}{n}{0}$ viewed as a $100$-regular sequence.
The sum matrix $M_{s,100}=M'_0 + \cdots + M'_{99}$ has eigenvalues $\beta=49 + 20\sqrt{6}$ and $49 - 20 \sqrt{6}$. As $\alpha^2 =\beta$, the combination of~\cref{thm: Drmota and Grabner,lem:Titchmarsh} (in base $100$) also gives~\cref{thm:abscissa-convergence-blocks-12-89}.
\qed
\end{proof}

A reasoning similar to what precedes can be conducted to obtain an alternative proof of~\cref{thm:abscissa-convergence-blocks-12-21}.
Let us now consider yet a different example.
We let $L_5$ denote the language of all words $w_k w_{k-1} \cdots w_0$ over $[0,9]$ such that $w_k \neq 0$ and, for all $i\in[0,k-1]$, 
\[
w_{i+1}w_i \neq 
\begin{cases} 
12 \text{ or } 89, & \text{if $i$ is even};\\
89, & \text{if $i$ is odd}.
\end{cases}
\]
%We define the sequence $\seq{t}{n}{0}$ by $t_n=1$ if $\rep_{10}(n)\in L_5$, $t_n=0$ otherwise.
The first values of $(|L_5 \cap [0,9]^n|)_{n\ge 0}$ are $1, 9, 88, 872, 8534, 84566, 827622$, which does not seem to be in the OEIS.
Computing further values of the latter sequence seems difficult with a classical computer set-up.
%The DFAO in~\cref{fig:automata-2} generates the $10$-automatic sequence $\seq{t}{n}{0}$ (and reads base-$10$ representations with least significant digit first).
%
%
%\begin{figure}
%    \centering
%    \includegraphics[scale=0.75]{automate-12-89even-89odd.pdf}
%    \caption{The base-$10$ DFAO generates the sequence $\seq{t}{n}{0}$ (read base-$10$ representations with least significant digit first).}
%    \label{fig:automata-2}
%\end{figure}
%
%
Following either approach (\cref{thm: Drmota and Grabner} with base $100$ or~\cite[Theorem 1]{Heuberger-Krenn-2020} with base $10$) combined with~\cref{lem:Titchmarsh} yields the following result.
%Following the approach of~\cite[Theorem A]{Heuberger-Krenn-2020} and from the DFAO in~\cref{fig:automata-2}, we deduce a linear representation of $\seq{u}{n}{0}$ viewed as a $10$-regular sequence, whose sum matrix has eigenvalues $\pm \sqrt{\frac{1}{2} \left(97+\sqrt{9401}\right)}$, $\pm \frac{2}{\sqrt{97+\sqrt{9401}}}$, and $0$ (each having multiplicity $1$).
%\[
%\pm \sqrt{\frac{1}{2} \left(97+\sqrt{9401}\right)},
%-\sqrt{\frac{1}{2} \left(97+\sqrt{9401}\right)},
%\frac{2}{\sqrt{97+\sqrt{9401}}},
%-\frac{2}{\sqrt{97+\sqrt{9401}}},
%0,
%\]
%Combining~\cref{lem:Titchmarsh} and~\cite[Theorem A]{Heuberger-Krenn-2020} gives us the following result.

\begin{theorem}
\label{thm:abscissa-convergence-blocks-12-89-even-12-odd}
    Let $\lambda=\sqrt{\frac{1}{2} \left(97+\sqrt{9401}\right)}$.
    Then the abscissa of convergence of the restricted Dirichlet 
    series $F_{L_5}(z)$ is $\frac{\log \lambda}{\log 10}$.
\end{theorem}

As already mentioned, the languages studied in this section have a characteristic sequence (in base $b$) that is $b$-automatic.
Similarly, given a subset $S\subseteq \mathbb{N}$, its \emph{characteristic sequence} $(s_n)_{n\ge 0}$ is defined by $s_n = 1$ if $n\in S$, $s_n=0$ otherwise.
When the latter is automatic, we obtain the following general result about the abscissa of convergence of the corresponding Dirichlet series.

\begin{theorem}
\label{thm: general abscissa of convergence automatic sec}
Let $S$ be a non-empty subset of integers such that its characteristic sequence $(s_n)_{n\ge 0}$ is $b$-automatic for some integer $b\ge 2$.
Let $L_S$ denote the language of $b$-representations of elements of $S$, i.e., $L_S=\{\rep_b(n) \mid n\in S\}$.
Then the abscissa of convergence of the corresponding Dirichlet series $F_{L_S}(z)$ is $0$, $1$ or a real number $s\in (0,1)$ such that $s=\frac{\log r}{\log b}$ for some integer $r>1$.
In particular, $r$ and $s$ can be explicitly computed. Furthermore, $s$ is always the quotient of two logarithms of integers.
\end{theorem}
\begin{proof}
Let $(A(n))_{n\ge 0}$ denote the summatory function of the sequence $(s_n)_{n\ge 0}$.
Combining the statements of~\cite[Theorems~11 and~12]{Cobham-1972} (for which we take the letter $a=1$), the summatory function $(A(n))_{n\ge 0}$ satisfies one of the following cases:
\begin{enumerate}
    \item[1)] $A(n) = \Theta(n)$;
    \item[2)] $A(n) = \Theta( n^s (\log n)^p)$ for some real $s\in(0,1)$ and some integer $p\ge 0$;
    \item[3)] $A(n) \sim c (\frac{\log n}{\log b})^p$ for some rational $c>0$ and some integer $p\ge 0$.
\end{enumerate}
Indeed, cases called $(i)$ and $(ii)$ in~\cite[Theorem~11]{Cobham-1972} lead to the first case above.
The last two cases are handled with~\cite[Theorem~12]{Cobham-1972} (note that, as we assume $S$ to be non-empty, the letter $a=1$ appears in $(s_n)_{n\ge 0}$, so we may discard case called $(iii)$ in~\cite[Theorem~12]{Cobham-1972}).
Now~\cref{lem:Titchmarsh} implies that the abscissa of convergence of $F_{L_S}(z)$ is $1$, $s\in(0,1)$, or $0$ depending on the case.
Now we refine the second case.
Indeed, examining the proof of case $(ii)$ in~\cite[Theorem~12]{Cobham-1972} shows that $s= \frac{\log r}{\log b}\in(0,1)$, where $r>1$ is the maximum modulus among all eigenvalues of some explicit matrix related to $(s_n)_{n\ge 0}$.
\qed
\end{proof}

\begin{example}
We obtain an alternative proof of~\cref{cor:avoid only one letter,cor:avoid only one letter bis} for $b\ge 3$ since, in this case, $r=b-1$ (also see~\cite[Example~5, p. 189]{Cobham-1972} for $b=3$).

If we consider the regular language $L_{P2}$ of base-$2$ representations of powers of $2$, then the corresponding summatory function satisfies the third item in the proof of~\cref{thm: general abscissa of convergence automatic sec} with $b=2$ and $p=1$~\cite[Example~3, p. 189]{Cobham-1972}.
The abscissa of convergence of the corresponding Dirichlet series $F_{L_{P2}}(z)$ is then equal to $0$.
\end{example}

We finish this section with the following observation on continuation of Dirichlet series.
%\begin{remark}
%\label{rk:meromorphic continuation}
Let $\seq{s}{n}{0}$ be a $b$-regular sequence whose linear representation is given by the vectors $V,W$ and the matrices $M_0,\ldots,M_{b-1}$.
Write $M_{s,b}= \sum_{i=0}^{b-1} M_i$.
Then one of the main results\footnote{In~\cite{AMFP}, the result is proved explicitly for $b$-automatic sequences and only stated for $b$-regular sequences. Also see~\cite[Theorem~3.1]{Coons2010}.} in~\cite{AMFP} asserts that the Dirichlet series $F_s(z)$ has a meromorphic continuation to the whole complex plane with \emph{candidate} poles located at
\[
z_{n,\ell}(\gamma) = \frac{\log \gamma}{\log b} - \ell + i \frac{2\pi n}{\log b}
\]
for all integers $n\in\mathbb{Z}$ and $\ell\ge 1$ and for each eigenvalue $\gamma$ of the sum matrix $M_{s,b}$ (note that $\log$ is the complex branch of the logarithm).
Proving that $z_{n,\ell}(\gamma)$ is actually a pole might turn out to be complicated.
However, in the following case, we are in power to say more.
Assume furthermore that $M_{s,b}$ is a primitive integer matrix and that the sequence $\seq{s}{n}{0}$ takes non-negative values.
If we let $\rho = \rho(M_{a,b})$ be the spectral radius of $M_{s,b}$, then the reasoning exposed in~\cite[Section~2]{Coons-Lind-2024} tells us that $\frac{\log \rho}{\log b}$ is a simple pole.

We may apply this observation to the examples of the present section.

\begin{example}
\label{ex: simple pole for all distinct blocks}
Considering the characteristic sequence $\seq{s}{n}{0}$ of the language $L_1$ from~\cref{sec:disjoint blocks} viewed as a $100$-regular sequence, the corresponding matrix $M_{s,100}$
%\[
%M_{s,100}
%=
%\left(
%\begin{array}{cc}
% 90 & 9 \\
% 81 & 9 \\
%\end{array}
%\right),
%\]
has integer coefficients, is primitive, and has eigenvalues $\beta=49 + 20\sqrt{6}$ and $49 - 20 \sqrt{6}$.
In particular, $F_{L_1}(z)$ has
\[
\frac{\log \rho(M_{s,100})}{\log 100} = \frac{\log \beta }{\log 100}=\frac{\log (5 + 2\sqrt{6})}{\log 10}
\]
as a simple pole.
We leave as an open problem to characterize the other poles of $F_{L_1}(z)$.
In addition, we ask: what are the poles of meromorphic continuations of Dirichlet series associated with block-avoiding binary automatic sequences?
The Dirichlet series having as coefficients the $(\pm 1)$-Thue--Morse sequence has no poles at all~\cite{Allouche-Cohen-1985}, while the Dirichlet series having as coefficients the $\{0,1\}$-Thue--Morse sequence has a single pole at $1$ which is simple (like the Riemann zeta function): namely, if $(t_n)_{n \geq 0}$ is the $\{0,1\}$-Thue--Morse sequence, then $(1 - 2 t_n)_{n \geq 0}$ is the $(\pm 1)$-Thue--Morse sequence.
\end{example}

\section{An unusual example}\label{sec:unusual}

In this section, we consider a last example which, as we will see, does not fit into the setting of $b$-regular sequences.
Let ${\bf t} = (t_n)_{n \geq 0} = 01101001\cdots$ be the $2$-automatic Thue-Morse sequence.
We say that an integer $n$ is \emph{evil} if $t_n = 0$.
A non-evil number is called \emph{odious}.
The first few evil numbers are 
\[
0, 3, 5, 6, 9, 10, 12, 15, 17, 18, 20;
\]
see also~\cite[A000069,A001969]{Sloane}.
Both sequences are $2$-regular~\cite{Allouche-Shallit-2003-ring2,Shallit-Walnut}.

Let $L_J$ be the set of binary words
of the form $w_k w_{k-1} \cdots w_0$ with the property that
there is no factor $10$ where the $0$ appears in an evil position $i$ for some $0\le i \le k$.
(Note that leading zeroes are allowed.)
For example, $10111\not\in L_J$ because the word $10$ appears
as $w_4 w_3$ and $3$ is evil.

Let us find a recurrence for the number $u_n= |L_J\cap \{0,1\}^n|$ of length-$n$ words
in $L_J$.
Here the Goulden-Jackson cluster method does not seem to be applicable.

\begin{lemma}
\label{lem:rec-for-bs}
We have $u_0 = 1$, $u_1 = 2$, $u_2 = 3$, and for all $n\ge 3$,
\begin{equation}
u_n = \begin{cases}
	2 u_{n-1}, & \text{if $t_{n-2} = 1$}; \\
	u_{n-1} + u_{n-3}, & \text{if $t_{n-2} = t_{n-3} = 0$}; \\
	u_{n-1} + u_{n-2}, & \text{if $t_{n-2} = 0$ and $t_{n-3} = 1$.}
	\end{cases}
\label{eq:foo}
\end{equation}
\end{lemma}

\begin{proof}
For $n\ge 0$, let $L_{J,n}= L_J\cap \{0,1\}^n$ denote the set of all words
$w_{n-1} \cdots w_0$ having no factor $10$ where the $0$
appears in an evil position, so that $u_n = |L_{J,n}|$.
Suppose $n \geq 3$.

\textbf{Case 1}: Assume that $t_{n-2} = 1$.  Then
$L_{J,n} = 0 L_{J,n-1} \, \cup \, 1 L_{J,n-1}$, because
appending either $0$ or $1$ to a word in $L_{J,n-1}$
cannot introduce a forbidden word.
This gives the first equality in~\eqref{eq:foo}.

\textbf{Case 2}:  Assume that $t_{n-2} = t_{n-3} = 0$.   Then
$L_{J,n} = 0 L_{J,n-1} \, \cup \, 111 L_{J,n-3}$.  The
first term in~\eqref{eq:foo} arises because
appending a $0$ to a word in $L_{J,n-1}$ cannot
introduce a forbidden word, and the second term
arises because the first three symbols of a word
in $L_{J,n}$ cannot be $100$, $101$, or $110$.
This gives the second equality in~\eqref{eq:foo}.

\textbf{Case 3}: Assume that $t_{n-2} = 0$ and $t_{n-3} = 1$.  Then
$L_{J,n} = 0 L_{J,n-1} \, \cup \, 11 L_{J,n-2}$.  The situation is like the
previous case, except now the first two symbols
of a word in $L_{J,n}$ cannot be $10$.
This ends the proof. \qed
\end{proof}

The first few values of the sequence $(u_n)_{n\ge 0}$ from~\cref{lem:rec-for-bs} are given in~\cref{tab:bn}.

\begin{table}
\begin{center}
\begin{tabular}{c|ccccccccccccccccccccc}
$n$ & 0 & 1 & 2 & 3 & 4 & 5 & 6 & 7 & 8 & 9 & 10 & 11 & 12 & 13 & 14 & 15 & 16 & 17 & 18 & 19 & 20\\
\hline
$u_n$ & 1 & 2 & 3 & 6 & 12 & 18 & 36 & 54 & 72 & 144 & 288 & 432 & 576 & 1152 & 1728 & 3456 & 6912 & 10368 & 20736 & 31104 & 41472
\end{tabular}
\end{center}
\caption{The number $u_n$ of length-$n$ words in $L_J$, i.e., words with the property that there is no factor $10$ where the $0$ appears in an evil position.}
\label{tab:bn}
\end{table}

We expect $u_n$ to behave roughly like $\alpha^n$, for some $\alpha$ with $1 < \alpha < 2$, because
forbidding $10$ with $0$ in an evil position is more or less like forbidding $10$ ``half the time''.  This intuition is confirmed in the following result.

\begin{theorem}
\label{thm: behavior of bn}
There are constants $c_1, c_2$ such that
$n^{c_1} \alpha^n \leq u_n \leq n^{c_2} \alpha^n$,
where $\alpha = \sqrt[6]{24} \approx 1.69838$.
\end{theorem}

To prove this, we require the next lemma.

\begin{lemma}
For $n \geq 3$, we have
\begin{equation}
u_n = \begin{cases}
	2 u_{n-1}, & \text{if $t_{n-2} = 1$}; \\
	\frac{4}{3} u_{n-1}, & \text{if $t_{n-2} = t_{n-3} = 0$}; \\
	\frac{3}{2} u_{n-1}, & \text{if $t_{n-2} = 0$ and $t_{n-3} = 1$.}
	\end{cases}
\label{eq:foo2}
\end{equation}
\end{lemma}

%JP I found that $u_n = h_n u_{n-1}$, where $u_n$
%is the $2$-regular sequence defined by:
%$h_n = (1/6) (4t_{n-2} + t_{n-3} -t_{n-2}t_{n-3} + 8)$.

\begin{proof}
The proof is by induction on $n$.  The base cases are left
to the reader.

\textbf{Case 1}: If $t_{n-2} = 1$, then the first equality in~\eqref{eq:foo2} follows immediately
from~\eqref{eq:foo}.

\textbf{Case 2}: Assume that $t_{n-2} = t_{n-3} = 0$.   Since $\bf t$
avoids cubes, we must have $t_{n-4} = 1$.
Then~\eqref{eq:foo} gives $u_n = u_{n-1} + u_{n-3}$ and $u_{n-1} = u_{n-2} + u_{n-3}$, and the induction hypothesis gives $u_{n-2} = 2 u_{n-3}$. Adding these last
two equations and canceling $u_{n-2}$ gives
$u_{n-1} = 3u_{n-3}$. The second equality in~\eqref{eq:foo2} now follows.

\textbf{Case 3}: Assume that $t_{n-2} = 0$ and $t_{n-3} = 1$. Then~\eqref{eq:foo} gives $u_n = u_{n-1} + u_{n-2}$ and the induction hypothesis gives $u_{n-1} = 2 u_{n-2}$.
The result now follows. \qed
\end{proof}

\begin{proof}[of~\cref{thm: behavior of bn}]
For a block (or pattern) $p \in \{0,1\}^*$ define $e_p (n)$ to be the number
of (possibly overlapping) occurrences of $p$ in ${\bf t}[0..n-1]$.
A simple induction, using \eqref{eq:foo2}, now gives
\begin{equation}
    u_n = 2^{e_1(n-1) + 2e_{00} (n-1) - e_{10} (n-1)} 3^{1+e_{10}(n-1) - e_{00} (n-1)}
    \label{eq:foo3}
\end{equation}
for $n \geq 2$.
For example, if $t_{n-2} = t_{n-3} = 0$, then ${\bf t}[0..n-2]={\bf t}[0..n-4]00$, so $e_1(n-1)=e_1(n-2)$, $e_{10} (n-1)=e_{10} (n-2)$, and $e_{00} (n-1)=e_{00} (n-2)+1$.
This together with the induction hypothesis and the second equality in~\eqref{eq:foo2} yields Equation~\eqref{eq:foo3}.
The other two cases can be handled in a similar fashion.

We now claim that
\begin{align}
e_1 (n) &= n/2 + O(1), \label{e1}\\
e_{00} (n) &= n/6 + O(\log n), \label{e00}\\
e_{10} (n) &= n/3 + O(\log n). \label{e10}
\end{align}
Eq.~\eqref{e1} is trivial.
To see Eq.~\eqref{e00}, an easy induction shows that 
$$
    e_{00}(2^i) = \frac{2^i - 3 - (-1)^i}{6},
$$
for all $i\ge 1$ and furthermore that the number of $00$'s in the length-$2^i$ prefix of the complement ${\overline{\bf t}}$ is $\frac{2^{i-1} + (-1)^i}{3}$ for all $i\ge 1$.
Write $n$ in binary as $\sum_{i=0}^{k} w_i 2^i$; then from our estimates we see that $e_{00} (n)$ is $\sum_{i=0}^{k} (w_i \frac{2^i}{6} + O(1))$, and the
estimate follows from the fact that $k \leq \log_2 n$.
A similar proof works for $e_{10}$ to establish~\cref{e10}.

These estimates now suffice to prove the claim of~\cref{thm: behavior of bn}.
\qed
\end{proof}

Define $L_J'$ to be the subset of $L_J$ that contains binary words not starting with $0$.
Define $(v_n)_{n \geq 0}$ to be the characteristic sequence of $L_J'$ in base $2$.
%For all $n\ge 0$, we let
%\[
%v_n = 
%\begin{cases}
%    1, & \text{if $\rep_{2}(n)\in L_J'$};\\
%    0, & \text{otherwise}.
%\end{cases}
%\]
The following observation particularly shows that one cannot use the alternative methods of~\cref{sec:alternative approach} to study the behavior of the summatory function of the sequence $(v_n)_{n \geq 0}$.

\begin{proposition}
\label{pro:non-auto}
    The sequences $(u_n)_{n\ge 0}$ and $(v_n)_{n \geq 0}$ are not $2$-regular.
\end{proposition}
\begin{proof}
Writing $u_n = 2^{x_n} 3^{y_n}$, we see from Equation~\eqref{eq:foo3} that $(x_n)_{n \geq 0}$
    and $(y_n)_{n \geq 0}$ are both $2$-regular sequences, but
    $\seq{u}{n}{0}$ is not (as it grows too quickly; see~\cite[Theorem 16.3.1]{Allouche-Shallit-2003}).

    On the other hand, the sequence $(v_n)_{n \geq 0}$ is
    not $2$-automatic. Consider those
    $n$ with base-$2$ representation of 
    the form $1 0 1^i$ for some $i\ge 0$, i.e.,
    $n = 3 \cdot 2^i - 1$.  Then
    $v_n = 1$ if and only if $i$ is odious, i.e., 
    $t_i = 1$.  By a well-known theorem (\cite[Theorem 5.5.2]{Allouche-Shallit-2003}),
    if $\seq{v}{n}{0}$ is $2$-automatic, then
    $(v_{3 \cdot 2^i - 1})_{i\ge 0}$ is ultimately periodic.
    But $v_{3 \cdot 2^i - 1} = t_i$, and ${\bf t} = (t_i)_{i \geq 0}$ cannot be ultimately periodic because it is overlap-free.
    As $(v_n)_{n \geq 0}$ takes finitely many values, it cannot be $2$-regular either.
    This ends the proof. \qed
\end{proof}

Nevertheless,  we can handle the behavior of the summatory function of $\seq{v}{n}{0}$, and thus the abscissa of convergence of the corresponding Dirichlet 
    series for the language $L_J'$, as follows.

\begin{corollary}
\label{cor:evil}
    Let $\alpha = \sqrt[6]{24}$.
    The abscissa of convergence of the restricted Dirichlet 
    series $F_{L_J'}(z)$ is $\frac{\log \alpha}{\log 2}$.    
    \end{corollary}

\begin{proof} 
For all $n\ge 0$, we let $A(n)= \sum_{0\le i\le n} v_i$.
By~\cref{thm: behavior of bn}, since leading zeroes are allowed in $L_J$, we have $n^{c_1} \alpha^n \leq A(2^n-1) \leq n^{c_2} \alpha^n$.
Now we may apply~\cref{le:useful} first, then~\cref{lem:Titchmarsh} to obtain the statement. \qed
\end{proof}

\section{Conclusion}
\label{sec:conclusion}

We start this conclusion by noting that the sequences of coefficients of the Dirichlet series studied in \cite{Kohler-Spilker-2009,Nathanson2021} are binary automatic sequences, as well as those presented in~\cref{sec:myriad,sec:alternative approach}.
They are in fact related to languages whose characteristic sequences are automatic.
It is worth briefly explaining why {\em automatic} Dirichlet series are interesting.
Recall from~\cref{sec:notation and def} that a sequence $(s_n)_n$ is $b$-automatic for some $b \geq 2$ if its $b$-kernel is finite.
This finiteness implies equalities between subsequences of the $b$-kernel, which in turn yield properties (e.g., functional properties) of the corresponding Dirichlet series, hence of the Dirichlet series we started from.
A similar conclusion on Dirichlet series might be drawn in the case of $b$-regular sequences.
This observation has been used in different contexts (see, e.g., \cite{Alkauskas-2004,AMFP,Allouche-Cohen-1985,Coons2010,Coons-Lind-2024,Everlove-2002,Flajolet-etal-1994,Toth-2022} and~\cite[Section~4.2.6]{Queffelec-Queffelec-2020}). (We note that in~\cite{Alkauskas-2004} particular completely multiplicative functions are considered, which are $b$-regular as noted in~\cite{Drmota-Grabner}, while~\cite{Everlove-2002} focuses on sum-of-digits functions, which form a typical example of $b$-regular sequences~\cite{Allouche-Shallit-2003}, and~\cite{Allouche-Cohen-1985,Toth-2022} concern the Thue--Morse sequence).
The Dirichlet series we are interested in here are related to languages where some patterns are avoided, an active research topic in combinatorics on words.
Another class of Dirichlet series related to combinatorics on words that has been studied consists of Dirichlet series with Sturmian coefficients (see, e.g., \cite{Kwon2015,Sourmelidis2019}).
A class of Dirichlet series more general than the Sturmian series was studied in \cite{Knill-Lesieutre}.

In this paper, we have chosen to give various methods and examples to examine the abscissa of convergence of particular Dirichlet series.
All our examples are related to counting words with missing digits or having some prescribed structure belonging to some language $L$.
In the first part of the paper (up to~\cref{sec:alternative approach}), the Dirichlet series $F_L(z)$ turns out to have automatic coefficients and we present a range of methods to handle this case, as well as a general statement with~\cref{thm: general abscissa of convergence automatic sec}.
Nonetheless, the methods presented in~\cref{sec:alternative approach}, although powerful, do not cover the case of non-regular Dirichlet series.
This is the topic of~\cref{sec:unusual} where we exhibit a language $L$ whose characteristic sequence $(s_L(n))_{n\ge 0}$ is not $b$-regular.
There, we use the power of~\cref{lem:Titchmarsh} that states that the abscissa of convergence of $F_L(z)$ is given by the growth of the summatory function of $(s_L(n))_{n\ge 0}$.
The key to obtain the abscissa of convergence is therefore to \emph{count} the words in $L$ (as we did in~\cref{sec:myriad}).
We mention that, to do this, other approaches than ours can be used (we think, e.g., of \cite{Guibas-Odlyzko-81}), and we have already mentioned the ``Goulden-Jackson cluster method'' and the paper \cite{Noonan-Zeilberger-1999}.

From the various results of the paper, we observe that the abscissas of convergence of our Dirichlet series all involve the quotient of the logarithms of two integers.
We thus leave the following as an open problem: for which sequences $\seq{s}{n}{1}$, does the corresponding Dirichlet series $F_s(z)$ have an abscissa of convergence that is the quotient of the logarithms of two integers?
Similarly, do the cases we have described here call for a general theorem and how general such a theorem can be?
We have also mentioned, in~\cref{ex: simple pole for all distinct blocks}, an
open problem related to~\cite{Everlove-2002}.

\subsection*{Acknowledgments}
Manon Stipulanti is an FNRS Research Associate supported by the Research grant 1.C.104.24F.
Jeffrey Shallit's research was supported by NSERC grants 2018-04118 and 2024-03725.

We thank France Gheeraert, Daniel Krenn, Eric Rowland, and Pierre Stas for useful 
discussions when this work was in its early stages.

We also thank the anonymous referees whose careful reading helped improving the paper, especially for enriching~\cref{sec:alternative approach} (starting from~\cref{thm: general abscissa of convergence automatic sec}), and for
correcting an error in our application of the Goulden-Jackson method.

%Eric's idea: assign an index modulo 2 to letters: $a_0$, $a_1$ and then apply the Goulden-Jackson method

%%%%%%%%%%%%%%%%%%%%%%%%%%%%%%%
%%%%%%%%%%%%%%%%%%%%%%%%%%%%%%%
%%%%%%      BIBLIO       %%%%%%
%%%%%%%%%%%%%%%%%%%%%%%%%%%%%%%
%%%%%%%%%%%%%%%%%%%%%%%%%%%%%%%

\bibliographystyle{splncs04}
\bibliography{biblio.bib}

\begin{thebibliography}{10}
\providecommand{\url}[1]{\texttt{#1}}
\providecommand{\urlprefix}{URL }
\providecommand{\doi}[1]{https://doi.org/#1}

\bibitem{Alkauskas-2004}
Alkauskas, G.: Dirichlet series associated with strongly {$q$}-multiplicative
  functions. Ramanujan J.  \textbf{8}(1),  13--21 (2004).
  \doi{10.1023/B:RAMA.0000027195.05101.2d},
  \url{https://doi.org/10.1023/B:RAMA.0000027195.05101.2d}

\bibitem{Allouche-Cohen-1985}
Allouche, {\relax J.-P}., Cohen, H.: Dirichlet series and curious infinite
  products. Bull. London Math. Soc.  \textbf{17}(6) (1985).
  \doi{10.1112/blms/17.6.531}, \url{https://doi.org/10.1112/blms/17.6.531}

\bibitem{Allouche-Yu-Morin-2024}
Allouche, {\relax J.-P}., Hu, Y., Morin, C.: Ellipsephic harmonic series
  revisited. Acta Math. Hungar.  \textbf{173}(2),  461--470 (2024).
  \doi{10.1007/s10474-024-01448-5},
  \url{https://doi.org/10.1007/s10474-024-01448-5}

\bibitem{AMFP}
Allouche, {\relax J.-P}., Mend\`es~France, M., Peyri\`ere, J.: Automatic
  {D}irichlet series. J. Number Theory  \textbf{81}(2),  359--373 (2000).
  \doi{10.1006/jnth.1999.2487}, \url{https://doi.org/10.1006/jnth.1999.2487}

\bibitem{Allouche-Shallit-Skordev-2005}
Allouche, {\relax J.-P}., Shallit, J., Skordev, G.: Self-generating sets,
  integers with missing blocks, and substitutions. Discrete Math.
  \textbf{292}(1-3),  1--15 (2005). \doi{10.1016/j.disc.2004.12.004}

\bibitem{Allouche-Shallit-1992}
Allouche, {\relax J.-P}., Shallit, J.: The ring of {$k$}-regular sequences.
  Theoret. Comput. Sci.  \textbf{98}(2),  163--197 (1992).
  \doi{10.1016/0304-3975(92)90001-V},
  \url{https://doi.org/10.1016/0304-3975(92)90001-V}

\bibitem{Allouche-Shallit-2003}
Allouche, {\relax J.-P}., Shallit, J.: Automatic sequences. Theory,
  Applications, Generalizations. Cambridge University Press, Cambridge (2003).
  \doi{10.1017/CBO9780511546563}

\bibitem{Allouche-Shallit-2003-ring2}
Allouche, {\relax J.-P}., Shallit, J.: The ring of {$k$}-regular sequences.
  {II}. Theoret. Comput. Sci.  \textbf{307}(1),  3--29 (2003).
  \doi{10.1016/S0304-3975(03)00090-2},
  \url{https://doi.org/10.1016/S0304-3975(03)00090-2}

\bibitem{CANT2010}
Berth\'{e}, V., Rigo, M. (eds.): Combinatorics, automata and number theory,
  Encyclopedia of Mathematics and its Applications, vol.~135. Cambridge
  University Press, Cambridge (2010). \doi{10.1017/CBO9780511777653}

\bibitem{Birmajer-Gil-Weiner-2016}
Birmajer, D., Gil, J.B., Weiner, M.D.: On the enumeration of restricted words
  over a finite alphabet. J. Integer Seq.  \textbf{19}(1),  Article 16.1.3, 16
  (2016), \url{https://cs.uwaterloo.ca/journals/JIS/VOL19/Gil/gil6.html}

\bibitem{Brauer-1951}
Brauer, A.: On algebraic equations with all but one root in the interior of the
  unit circle. Math. Nachr.  \textbf{4},  250--257 (1951).
  \doi{10.1002/mana.3210040123}

\bibitem{de-Bruijn-Knuth-Rice-1972}
de~Bruijn, N.G., Knuth, D.E., Rice, S.O.: The average height of planted plane
  trees. In: Graph theory and computing, pp. 15--22. Academic Press, New York
  (1972)

\bibitem{Burnol-3}
Burnol, {\relax J.-F}.: Digamma function and general fischer series in the
  theory of kempner sums. Expositiones Mathematicae  \textbf{42}(6),  125604
  (2024). \doi{https://doi.org/10.1016/j.exmath.2024.125604},
  \url{https://www.sciencedirect.com/science/article/pii/S0723086924000719}

\bibitem{Burnol-1}
Burnol, {\relax J.-F}.: Measures for the summation of {I}rwin series (2024),
  preprint available at \url{https://arxiv.org/abs/2402.09083}

\bibitem{Burnol-2}
Burnol, {\relax J.-F}.: Sur l'asymptotique des sommes de {K}empner pour de
  grandes bases (2024), preprint available at
  \url{https://arxiv.org/abs/2403.01957}

\bibitem{Cobham-1972}
Cobham, A.: Uniform tag sequences. Math. Systems Theory  \textbf{6},  164--192
  (1972). \doi{10.1007/BF01706087}, \url{https://doi.org/10.1007/BF01706087}

\bibitem{Coons2010}
Coons, M.: ({N}on)automaticity of number theoretic functions. J. Th\'{e}or.
  Nombres Bordeaux  \textbf{22}(2),  339--352 (2010),
  \url{http://jtnb.cedram.org/item?id=JTNB_2010__22_2_339_0}

\bibitem{Coons-Lind-2024}
Coons, M., Lind, J.: Radial asymptotics of generating functions of $k$-regular
  sequences. Bull. Aust. Math. Soc.  \textbf{110}(3),  528--534 (2024).
  \doi{10.1017/S0004972724000480},
  \url{https://doi.org/10.1017/S0004972724000480}

\bibitem{Cumberbatch}
Cumberbatch, J.: Digitally restricted sets and the {G}oldbach conjecture: An
  exceptional set result (2024), preprint available at
  \url{https://arxiv.org/abs/2402.07921}

\bibitem{Drmota-Grabner}
Drmota, M., Grabner, P.J.: Analysis of digital functions and applications. In:
  Combinatorics, automata and number theory, Encyclopedia Math. Appl.,
  vol.~135, pp. 452--504. Cambridge Univ. Press, Cambridge (2010)

\bibitem{Dumas-2013}
Dumas, P.: Joint spectral radius, dilation equations, and asymptotic behavior
  of radix-rational sequences. Linear Algebra Appl.  \textbf{438}(5),
  2107--2126 (2013). \doi{10.1016/j.laa.2012.10.013}

\bibitem{Dumas-2014}
Dumas, P.: Asymptotic expansions for linear homogeneous divide-and-conquer
  recurrences: algebraic and analytic approaches collated. Theor. Comput. Sci.
  \textbf{548},  25--53 (2014). \doi{10.1016/j.tcs.2014.06.036}

\bibitem{Dumas-2007}
Dumas, P., Lipmaa, H., Wall{\'e}n, J.: Asymptotic behaviour of a
  non-commutative rational series with a nonnegative linear representation.
  Discrete Math. Theor. Comput. Sci.  \textbf{9}(1),  247--272 (2007),
  \url{www.dmtcs.org/dmtcs-ojs/index.php/dmtcs/issue/view/85/showToc.html}

\bibitem{Dumont-Thomas}
Dumont, {\relax J.-M}., Thomas, A.: Systèmes de numération et fonctions
  fracta\-les relatifs aux substitutions. Theoret. Comput. Sci.
  \textbf{65}(2),  153--169 (1989). \doi{10.1016/0304-3975(89)90041-8}

\bibitem{Zeilberger2023}
Ekhad, S.B., Zeilberger, D.: Counting clean words according to the number of
  their clean neighbors. ACM Commun. Comput. Algebra  \textbf{57}(1), ~5--9
  (2023). \doi{10.1145/3610377.3610379}

\bibitem{Everlove-2002}
Everlove, C.: Dirichlet series associated to sum-of-digits functions. Int. J.
  Number Theory  \textbf{18}(4),  777--798 (2022).
  \doi{10.1142/S1793042122500415},
  \url{https://doi.org/10.1142/S1793042122500415}

\bibitem{Flajolet-etal-1994}
Flajolet, P., Grabner, P., Kirschenhofer, P., Prodinger, H., Tichy, R.F.:
  Mellin transforms and asymptotics: digital sums. Theoret. Comput. Sci.
  \textbf{123}(2),  291--314 (1994). \doi{10.1016/0304-3975(92)00065-Y},
  \url{https://doi.org/10.1016/0304-3975(92)00065-Y}

\bibitem{Flajolet-Sedgewick-2009}
Flajolet, P., Sedgewick, R.: Analytic combinatorics. Cambridge University
  Press, Cambridge (2009). \doi{10.1017/CBO9780511801655}

\bibitem{Fogg-2002}
Fogg, N.P.: Substitutions in dynamics, arithmetics and combinatorics, Lecture
  Notes in Mathematics, vol.~1794. Springer-Verlag, Berlin (2002).
  \doi{10.1007/b13861}, edited by V. Berth\'{e}, S. Ferenczi, C. Mauduit and A.
  Siegel

\bibitem{Graham-Knuth-Patashnik-1994}
Graham, R.L., Knuth, D.E., Patashnik, O.: Concrete mathematics. Addison-Wesley
  Publishing Company, Reading, MA, second edn. (1994), a foundation for
  computer science

\bibitem{Guibas-Odlyzko-81}
Guibas, L.J., Odlyzko, A.M.: String overlaps, pattern matching, and
  nontransitive games. J. Combin. Theory Ser. A  \textbf{30}(2),  183--208
  (1981). \doi{10.1016/0097-3165(81)90005-4},
  \url{https://doi.org/10.1016/0097-3165(81)90005-4}

\bibitem{Haeseler-2003}
von Haeseler, F.: Automatic sequences, De Gruyter Expositions in Math.,
  vol.~36. Walter de Gruyter \& Co., Berlin (2003). \doi{10.1515/9783110197969}

\bibitem{Heuberger-Krenn-2020}
Heuberger, C., Krenn, D.: Asymptotic analysis of regular sequences.
  Algorithmica  \textbf{82}(3),  429--508 (2020).
  \doi{10.1007/s00453-019-00631-3}

\bibitem{Janjic2015}
Janji\'{c}, M.: On linear recurrence equations arising from compositions of
  positive integers. J. Integer Seq.  \textbf{18}(4),  Article 15.4.7, 14
  (2015), \url{https://cs.uwaterloo.ca/journals/JIS/VOL18/Janjic/janjic63.html}

\bibitem{Janjic2016}
Janji\'{c}, M.: Binomial coefficients and enumeration of restricted words. J.
  Integer Seq.  \textbf{19}(7),  Article 16.7.3, 18 (2016),
  \url{https://cs.uwaterloo.ca/journals/JIS/VOL19/Janjic/janjic73.html}

\bibitem{Janjic2017}
Janji\'{c}, M.: Some formulas for numbers of restricted words. J. Integer Seq.
  \textbf{20}(6),  Art. 17.6.5, 10 (2017),
  \url{https://cs.uwaterloo.ca/journals/JIS/VOL20/Janjic/janjic83.html}

\bibitem{Janjic2018-2}
Janji\'{c}, M.: Pascal matrices and restricted words. J. Integer Seq.
  \textbf{21}(5),  Art. 18.5.2, 13 (2018),
  \url{https://cs.uwaterloo.ca/journals/JIS/VOL21/Janjic2/janjic103.html}

\bibitem{Janjic2018-1}
Janji\'{c}, M.: Words and linear recurrences. J. Integer Seq.  \textbf{21}(1),
  Art. 18.1.4, 17 (2018),
  \url{https://cs.uwaterloo.ca/journals/JIS/VOL21/Janjic/janjic93.html}

\bibitem{Karki-Lacroix-Rigo-2009}
K\"{a}rki, T., Lacroix, A., Rigo, M.: On the recognizability of self-generating
  sets. In: Mathematical foundations of computer science 2009, Lecture Notes in
  Comput. Sci., vol.~5734, pp. 525--536. Springer, Berlin (2009).
  \doi{10.1007/978-3-642-03816-7\_45}

\bibitem{Karki-Lacroix-Rigo-2010}
K\"{a}rki, T., Lacroix, A., Rigo, M.: On the recognizability of self-generating
  sets. J. Integer Seq.  \textbf{13}(2),  Article 10.2.2, 18 (2010),
  \url{https://cs.uwaterloo.ca/journals/JIS/VOL13/Rigo/rigo6.html}

\bibitem{Kempner1914}
Kempner, A.J.: A curious convergent series. Amer. Math. Monthly
  \textbf{21}(2),  48--50 (1914). \doi{10.2307/2972074}

\bibitem{Knill-Lesieutre}
Knill, O., Lesieutre, J.: Analytic continuation of {D}irichlet series with
  almost periodic coefficients. Complex Anal. Oper. Theory  \textbf{6}(1),
  237--255 (2012). \doi{10.1007/s11785-010-0064-7},
  \url{https://doi.org/10.1007/s11785-010-0064-7}

\bibitem{Kohler-Spilker-2009}
K\"{o}hler, G., Spilker, J.: Dirichlet-{R}eihen zu {K}empners merkw\"{u}rdiger
  konvergenter {R}eihe. Math. Semesterber.  \textbf{56}(2),  187--199 (2009).
  \doi{10.1007/s00591-009-0059-5}

\bibitem{Kwon2015}
Kwon, D.: A one-parameter family of {D}irichlet series whose coefficients are
  {S}turmian words. J. Number Theory  \textbf{147},  824--835 (2015).
  \doi{10.1016/j.jnt.2014.08.018},
  \url{https://doi.org/10.1016/j.jnt.2014.08.018}

\bibitem{Mandelbrojt}
Mandelbrojt, S.: Dirichlet {S}eries. {P}rinciples and {M}ethods. D. Reidel
  Publishing Co., Dordrecht (1972)

\bibitem{MukherjeeSarkar}
Mukherjee, R., Sarkar, N.: A short note on a curious convergent series.
  Asian-Eur. J. Math.  \textbf{14}(9),  Paper No. 2150158 (2021).
  \doi{10.1142/S1793557121501588}

\bibitem{Nathanson2021}
Nathanson, M.B.: Dirichlet series of integers with missing digits. J. Number
  Theory  \textbf{222},  30--37 (2021). \doi{10.1016/j.jnt.2020.10.002}

\bibitem{Noonan-Zeilberger-1999}
Noonan, J., Zeilberger, D.: The {Goulden}-{Jackson} cluster method:
  {Extensions}, applications and implementations. J. Difference Equ. Appl.
  \textbf{5}(4-5),  355--377 (1999). \doi{10.1080/10236199908808197}

\bibitem{Queffelec-Queffelec-2020}
Queffélec, H., Queffélec, M.: Diophantine approximation and {D}irichlet
  series, Texts and Readings in Mathematics, vol.~80. Hindustan Book Agency,
  New Delhi; Springer, Singapore, second edn. (2020).
  \doi{10.1007/978-981-15-9351-2},
  \url{https://doi.org/10.1007/978-981-15-9351-2}

\bibitem{Rowland18Package-1}
Rowland, E.: {\textsc{IntegerSequences}} software package,
  \url{https://github.com/ericrowland/IntegerSequences}

\bibitem{Rowland18Package-2}
Rowland, E.: \textsc{IntegerSequences}: a package for computing with
  {{\(k\)}}-regular sequences. In: Mathematical software -- ICMS 2018. 6th
  international conference, South Bend, IN, USA, July 24--27, 2018.
  Proceedings, pp. 414--421. Cham: Springer (2018).
  \doi{10.1007/978-3-319-96418-8\_49}

\bibitem{Shallit-Walnut}
Shallit, J.: The logical approach to automatic sequences: exploring
  combinatorics on words with {\tt {W}alnut}, London Mathematical Society
  Lecture Note Series, vol.~482. Cambridge University Press, Cambridge (2023)

\bibitem{Sloane}
Sloane, N.J.A., et~al.: The {O}n-{L}ine {E}ncyclopedia of {I}nteger
  {S}equences, \url{https://oeis.org}

\bibitem{Sourmelidis2019}
Sourmelidis, A.: On the meromorphic continuation of {B}eatty zeta-functions and
  {S}turmian {D}irichlet series. J. Number Theory  \textbf{194},  303--318
  (2019). \doi{10.1016/j.jnt.2018.07.009},
  \url{https://doi.org/10.1016/j.jnt.2018.07.009}

\bibitem{Titchmarsh-1939}
Titchmarsh, E.C.: The theory of functions. Oxford University Press, Oxford,
  second edn. (1939)

\bibitem{Toth-2022}
T\'oth, L.: Linear combinations of {D}irichlet series associated with the
  {T}hue-{M}orse sequence. Integers  \textbf{22},  Paper No. A98, 9 (2022)

\bibitem{Wilf-2006}
Wilf, H.S.: {G}eneratingfunctionology. A K Peters, Ltd., Wellesley, MA, third
  edn. (2006),
  \url{https://web.archive.org/web/20230322192727/https://www2.math.upenn.edu/~wilf/DownldGF.html}

\bibitem{Wintner-1947}
Wintner, A.: On {R}iemann's reduction of {D}irichlet series to power series.
  Amer. J. Math.  \textbf{69},  769--789 (1947). \doi{10.2307/2371798}

\end{thebibliography}
\end{document}